\documentclass[12pt, reqno]{amsart}
\usepackage{dsfont}
\usepackage{enumerate}
\usepackage{amsfonts}
\usepackage{amssymb}
\usepackage{amsthm}

\newtheorem{theorem}{Theorem}[section]
\newtheorem{proposition}[theorem]{Proposition}
\newtheorem{lemma}[theorem]{Lemma}
\newtheorem{corollary}[theorem]{Corollary}
\theoremstyle{definition}

\newtheorem{definition}[theorem]{Definition}

\newtheorem{remark} [theorem] {Remark}

\begin{document}
\title[Operators commuting with $L^\infty$]{Operators commuting with $L^\infty$ on K\"othe--Bochner spaces}

\author{Hülya Duru}

\address{Department of Mathematics, Faculty of Science, Istanbul University, Vezneciler-Istanbul, 34134, Turkey}

\email{hduru@istanbul.edu.tr}

\author{Serkan Ilter}

\address{Department of Mathematics, Faculty of Science, Istanbul University, Vezneciler-Istanbul, 34134, Turkey}

\email{ilters@istanbul.edu.tr}

\author{Arkady Kitover}

\address{Community College of Philadelphia, 1700 Spring Garden St., Philadelphia, PA 19130, USA}

\email{akitover@ccp.edu}

\author{Mehmet Orhon}

\address{University of New Hampshire, 105 Main Street
Durham, NH 03824, USA}

\email{mo@unh.edu}


\date{\today}


\begin{abstract}
Let $E(X)$ be a K\"othe--Bochner space of functions on a $\sigma$-finite measure space with values in a Banach space $X$. We study bounded operators on $E(X)$ that commute with the operators of multiplication by functions from $L^\infty$. If $X$ is separable, every such operator acts pointwise by an essentially bounded $L(X)$-valued function. For a nonseparable $X$, a lifting of $L^\infty$ provides a representation by $L(X,X^{\star\star})$-valued functions, and by $L(X)$-valued functions if $X$ is complemented in $X^{\star\star}$, with the sharp norm bound if $X$ is a dual space. As an application, we show that if the transformation is aperiodic, then the spectral projections of invertible weighted composition operators with operator-valued weights are operators of multiplication by $L(X)$-valued functions whenever $X$ is separable or complemented in $X^{\star\star}$, and that, if moreover the measure is ergodic and $\dim X = m < \infty$, the spectrum of such an operator has at most $m$ connected components. Finally, for $L^p(\mu;X)$, $1 \leq p < \infty$, an aperiodic nonsingular transformation and a weight that is measurable in the operator norm, we prove that the spectral projections are operators of multiplication by projection-valued functions for an arbitrary Banach space $X$, for instance for $X = c_0(\Gamma)$ with $\Gamma$ uncountable. The proof uses an inverse-closedness result for the algebra generated by norm-measurable multipliers and the composition operator.
\end{abstract}

\maketitle

\markboth{Duru, Ilter, Kitover, Orhon}{Operators commuting with $L^\infty$}

\setcounter{section}{-1}
\section{Introduction}

Let $(\Omega, \Sigma, \mu)$ be a complete $\sigma$-finite measure space, let $E$ be a Banach function (K\"othe) space on $\Omega$, let $X$ be a Banach space, and let $E(X)$ be the corresponding K\"othe--Bochner space of (classes of) strongly measurable functions $x : \Omega \to X$ with $\|x(\cdot)\|_X \in E$. Every $f \in L^\infty(\mu)$ defines a bounded operator $x \mapsto fx$ on $E(X)$. In this paper we study the commutant of this algebra of operators. The basic question is whether every bounded operator $T$ on $E(X)$ that commutes with $L^\infty(\mu)$ is \emph{decomposable}, that is, whether there is a function $A : \Omega \to L(X)$ such that
\begin{equation} \label{eqintro}
(Tx)(\omega) = A(\omega)x(\omega) \quad \text{for almost all } \omega \in \Omega, \qquad x \in E(X).
\end{equation}

For $X = \mathds{C}$ the answer is well known: the commutant of $L^\infty(\mu)$ in $L(E)$ is $L^\infty(\mu)$ itself, i.e., the operators of multiplication by functions from $L^\infty(\mu)$ form a maximal commutative subalgebra of $L(E)$ (this is also the case $X = \mathds{C}$ of Theorem~\ref{t1} below). For $E = L^2(\mu)$ and a separable Hilbert space $X$ the answer is given by von Neumann's reduction theory \cite{vN49}: the commutant of the diagonal algebra $L^\infty(\mu)$ acting on $L^2(\mu; X)$ consists exactly of the decomposable operators (\ref{eqintro}), where $A$ is a bounded weakly measurable $L(X)$-valued function and $\|T\| = \operatorname{ess\,sup}_\omega \|A(\omega)\|$ (see also \cite[Part II, Chapter 2]{Di}).

In the setting of Banach function spaces the closely related problem of characterizing the \emph{scalar} multiplication operators has attracted attention. J.~M.~Calabuig, J.~Rodr\'{\i}guez and E.~A.~S\'anchez-P\'erez \cite{CRS} showed that, if $E$ has order continuous norm, an operator on $E(X)$ is an operator of multiplication by a function from $L^\infty(\mu)$ if and only if it satisfies a certain functional equation. In \cite{DKO} we extended their result to arbitrary K\"othe--Bochner spaces and showed that $T$ is such an operator if and only if it commutes with $L^\infty(\mu)$ and leaves invariant the cyclic subspaces generated by the constant functions. Algebras of multiplication operators in Banach function spaces and their bicommutants were studied by B.~de~Pagter and W.~J.~Ricker \cite{dPR}. For spaces of continuous functions, operators commuting with $C(K)$ on $C(K,X)$ are described by the theory of $M$-structure and Banach $C(K)$-modules \cite{Be, AAK}; in that case they are always given by strongly continuous $L(X)$-valued functions (see Remark~\ref{r2}).

Our first results concern the representation (\ref{eqintro}) in general K\"othe--Bochner spaces. If $X$ is separable, every $T$ commuting with $L^\infty(\mu)$ has the form (\ref{eqintro}) with $\sup_\omega \|A(\omega)\| \leq \|T\|$ (Theorem~\ref{t1}). If $X$ is not separable, the countable-dense-set argument fails, and we use instead a lifting of $L^\infty(\mu)$, whose existence goes back to J.~von Neumann \cite{vN31} and D.~Maharam \cite{Ma}. In general the lifting provides a representation (\ref{eqintro}) with $A(\omega) \in L(X, X^{\star\star})$ (Remark~\ref{r1}). If $X$ is complemented in $X^{\star\star}$ by a projection $Q$, we obtain $A(\omega) \in L(X)$ with $\|T\| \leq \sup_\omega \|A(\omega)\| \leq \|Q\|\,\|T\|$ (Theorem~\ref{t2}). In particular, the representation holds with the sharp bound $\|T\|$ for every dual and every reflexive space (Corollary~\ref{c4}). If $\mu$ is purely atomic, (\ref{eqintro}) holds for every Banach space $X$ (Remark~\ref{ratomic}), so the difficulties come entirely from the nonatomic part of $\mu$.

The second part of the paper deals with weighted composition operators $(Tx)(\omega) = a(\omega)x(\varphi(\omega))$ with operator-valued weights. Spectral theory of such operators, and in particular the connection between their spectral projections and hyperbolicity (exponential dichotomy) of the associated linear extensions of dynamical systems, is the subject of an extensive literature; we refer to the books of A.~Antonevich \cite{An}, A.~Antonevich and A.~Lebedev \cite{AL}, and C.~Chicone and Y.~Latushkin \cite{CL}, to the survey of Yu.~Latushkin and A.~Stepin \cite{LS}, and to \cite{AAK}. If $T$ is invertible and $\varphi$ is aperiodic, the spectral projections of $T$ corresponding to arbitrary clopen subsets of $\sigma(T)$ commute with $L^\infty(\mu)$. Therefore, if $X$ is separable or complemented in $X^{\star\star}$, they are decomposable operators (Corollaries~\ref{c1} and \ref{c1a}). If moreover $\mu$ is ergodic with respect to $\varphi$ and $\dim X = m < \infty$, then $\sigma(T)$ has at most $m$ connected components (Corollary~\ref{c2}); this result and \cite[Theorem 9.5]{An} complement each other (Remark~\ref{r0}).

Finally, in Section~\ref{s4} we show that for $E = L^p(\mu)$, $1 \leq p < \infty$, an aperiodic nonsingular transformation $\varphi$, and a weight that is measurable in the operator norm (after the normalization that makes the composition operator an isometry), the spectral projections of $T$ are decomposable for an \emph{arbitrary} Banach space $X$, with pointwise projections $P(\omega)$ satisfying $\|P(\omega)\| \leq \|P\|$ (Theorem~\ref{t4} and Corollary~\ref{c313}). The main case of interest is that of a nonseparable space $X$ that is not complemented in $X^{\star\star}$, for example $X = c_0(\Gamma)$ with $\Gamma$ uncountable (Remark~\ref{r6}(4)). Here commutation with $L^\infty(\mu)$ is not enough, and we have to show that the spectral projection belongs to the closed algebra generated by the multiplication operators and the composition operator. The scheme of the proof goes back to Antonevich and Lebedev \cite{AL, An}; the main new ingredient is an inverse-closedness result for this (non-$C^\star$) algebra (Proposition~\ref{p310}), which is proved by restricting the operators to the orbits of $\varphi$.

\subsection*{Notation}
Throughout the paper we use the following notation.
\begin{itemize}
\item $\mathds{N}$, $\mathds{Z}$, $\mathds{R}$, $\mathds{C}$ are the sets of natural numbers, integers, real numbers and complex numbers; $\mathds{Q}(i) = \{p + iq : p, q \text{ rational}\}$ is the field of complex rational numbers.
\item $\mathds{T}$, $\mathds{U}$, and $\mathds{D}$ are the unit circle, the open unit disk, and the closed unit disk in $\mathds{C}$.
\item $(\Omega, \Sigma, \mu)$ is a complete measure space with a $\sigma$-finite measure $\mu$. To avoid trivialities we assume throughout that $\mu(\Omega) > 0$; if $\mu(\Omega) = 0$ then $E(X) = \{0\}$ and all the statements below are trivial; $L^0(\Omega, \Sigma, \mu)$ is the space of (classes of) measurable functions on $\Omega$; $L^\infty(\mu) = L^\infty(\Omega, \Sigma, \mu)$; $\mathcal{L}^\infty(\Omega, \Sigma, \mu)$ is the space of all bounded $\mu$-measurable functions (not classes) with the supremum norm.
\item $\chi_G$ is the indicator function of $G \in \Sigma$; for $\mu(G) < \infty$, $P_G$ is the operator of multiplication by $\chi_G$. A function $f \in L^\infty(\mu)$ is identified with the operator of multiplication by $f$, and we say that $T$ \emph{commutes with $L^\infty$} if $Tf = fT$ for all $f \in L^\infty(\mu)$.
\item $E \subset L^0(\Omega, \Sigma, \mu)$ is a Banach function (K\"othe) space such that $\chi_G \in E$ whenever $\mu(G) < \infty$.
\item $X$ and $Y$ are Banach spaces, $X^\star$ and $X^{\star\star}$ are the dual and the second dual of $X$, and $\langle x, x^\star \rangle$ is the value of $x^\star \in X^\star$ at $x \in X$. $L(X,Y)$ is the space of bounded linear operators from $X$ into $Y$, and $L(X) = L(X,X)$. $\sigma(T)$ is the spectrum of an operator $T$.
\item $E(X)$ is the K\"othe--Bochner space of (classes of) strongly measurable functions $x : \Omega \to X$ such that $\|x(\cdot)\| \in E$, with the norm $\|x\| = \big\| \|x(\cdot)\| \big\|_E$. For $x \in X$, $\vec{x}$ is the constant function $\vec{x}(\omega) \equiv x$.
\item $L^p(Y) = L^p(\Omega, \Sigma, \mu; Y)$, $1 \leq p < \infty$, and $\ell^p(\mathds{Z}; Y)$ is the space of $p$-summable two-sided sequences in $Y$.
\item $\varphi : \Omega \to \Omega$ is an invertible measurable map, $\varphi^n$ is its $n$-th iterate, $n \in \mathds{Z}$, and $(Vx)(\omega) = x(\varphi(\omega))$ is the corresponding composition operator; in Section~\ref{s4} the letter $V$ denotes the normalized composition operator (\ref{eqV}), which is an isometry of $L^p$. An operator of the form $(Tx)(\omega) = a(\omega)x(\varphi(\omega))$, where $a : \Omega \to L(X)$, is called a \emph{weighted composition operator}; the function $a$ is its \emph{weight}.
\item $\Pi_r$ denotes the spectral projection of an invertible weighted composition operator corresponding to the part of its spectrum inside the circle $\{|\lambda| = r\}$ (Sections~\ref{s3} and~\ref{s4}).
\item $L^\infty_u(\Omega; L(Y))$ is the space of essentially bounded functions $b : \Omega \to L(Y)$ that are measurable with respect to the operator norm (\emph{norm-measurable} functions), and $M_b$ is the operator $(M_bx)(\omega) = b(\omega)x(\omega)$ (Section~\ref{s4}).
\item The algebras $\mathcal{A}_u$, $\mathcal{B}_0$, $\mathcal{B}$, the map $N_0$, the orbit operators $\pi_\omega$, the shift $S$ on $\ell^p(\mathds{Z};Y)$, and the matrix entries $C[n,m]$ are defined in Section~\ref{s4}.
\end{itemize}

\section{The case of separable $X$}

 \begin{definition} \label{d1} Let us recall that a function $z \in E(X)$ is called $\sigma$-valued if there are a nonempty finite or countable set $I \subseteq \mathds{N}$, vectors $z_i \in X$, $i \in I$, and pairwise disjoint sets $G_i \in \Sigma$, $\mu(G_i) > 0$, $i \in I$, such that $\mu(\Omega \setminus \bigcup \limits_{i \in I} G_i) = 0$ and
  $\chi_{G_i}z = \chi_{G_i}\vec{z}_i$, $i \in I$.
  \end{definition}

The index set $I$ is allowed to be finite: if, for example, $\Omega$ consists of finitely many atoms, there is no infinite family of pairwise disjoint sets of positive measure.

The following lemma allows us to reduce the case of a $\sigma$-finite measure to the case of a probability measure. Notice that the reduction requires not only a change of the measure but also a change of the space $E$: if $\mu(\Omega) = \infty$ then, in general, $\chi_\Omega \not\in E$.

\begin{lemma} \label{l0}
Let $(\Omega, \Sigma, \mu)$ be a complete measure space with a $\sigma$-finite measure $\mu$ such that $\mu(\Omega) > 0$, let $E \subset L^0(\Omega, \Sigma, \mu)$ be a Banach function space such that $\chi_G \in E$ whenever $\mu(G) < \infty$, let $X$ be a Banach space, and let $T$ be a bounded linear operator on $E(X)$ such that $TP_G = P_GT$ for all $G \in \Sigma$ with $\mu(G) < \infty$. Then there exist
\begin{enumerate}[(i)]
  \item a probability measure $\nu$ on $\Sigma$ such that $\nu(G) = 0 \Leftrightarrow \mu(G) = 0$, $G \in \Sigma$; in particular, $(\Omega, \Sigma, \nu)$ is complete, $L^0(\Omega, \Sigma, \nu) = L^0(\Omega, \Sigma, \mu)$ and $L^\infty(\nu) = L^\infty(\mu)$;
  \item a function $u \in E$ such that $u(\omega) > 0$ for all $\omega \in \Omega$ and such that the Banach function space $E_u = \{f \in L^0 : uf \in E\}$, $\|f\|_{E_u} = \|uf\|_E$, contains $\chi_\Omega$,
\end{enumerate}
with the following properties. The map $J : E_u(X) \to E(X)$, $(Jy)(\omega) = u(\omega)y(\omega)$, is a surjective linear isometry, and the operator $\tilde{T} = J^{-1}TJ$ on $E_u(X)$ satisfies $\|\tilde{T}\| = \|T\|$ and $\tilde{T}P_G = P_G\tilde{T}$ for all $G \in \Sigma$. If $T$ commutes with $L^\infty$, then so does $\tilde{T}$.

Moreover, if $\bar{\Omega} \subset \Omega$, $\mu(\Omega \setminus \bar{\Omega}) = 0$, and $A : \bar{\Omega} \to L(X)$ are such that $\tilde{T}y(\omega) = A(\omega)y(\omega)$ for almost all $\omega$ for every $y \in E_u(X)$, then $Tx(\omega) = A(\omega)x(\omega)$ for almost all $\omega$ for every $x \in E(X)$.
\end{lemma}

\begin{proof}
Let $\Omega = \bigcup \limits_{n=1}^\infty \Omega_n$, where the sets $\Omega_n \in \Sigma$ are pairwise disjoint and $\mu(\Omega_n) < \infty$. Then $\chi_{\Omega_n} \in E$.

(i) Let $w = \sum \limits_{n=1}^\infty 2^{-n}(1 + \mu(\Omega_n))^{-1}\chi_{\Omega_n}$. Then $w(\omega) > 0$ for all $\omega \in \Omega$ and $0 < \int_\Omega w \, d\mu \leq 1$. Let $\nu(G) = \int_G w \, d\mu \big/ \int_\Omega w \, d\mu$, $G \in \Sigma$. Because $w > 0$ on $\Omega$, the measures $\mu$ and $\nu$ have the same sets of measure zero, and all the remaining statements of (i) follow.

(ii) Let $u = \sum \limits_{n=1}^\infty 2^{-n}(1 + \|\chi_{\Omega_n}\|_E)^{-1}\chi_{\Omega_n}$. The series converges absolutely in $E$, and since its partial sums increase and converge pointwise to $u$, we have $u \in E$ (recall that norm convergence in $E$ implies convergence in measure on sets of finite measure, see~\cite[Theorem~1, p.~96]{KA}). Clearly $u(\omega) > 0$ for all $\omega \in \Omega$. The map $f \mapsto uf$ is a linear isometry of $E_u$ onto $E$, whose inverse is $g \mapsto g/u$, and $|f_1| \leq |f_2|$ implies $|uf_1| \leq |uf_2|$. Therefore $E_u$ is a Banach function space on $(\Omega, \Sigma, \nu)$ (any additional lattice properties of $E$, such as the Fatou property or order continuity of the norm, are inherited by $E_u$), and $\chi_\Omega \in E_u$ because $u \in E$.

Because $u$ is a measurable, finite and strictly positive scalar function, a function $y : \Omega \to X$ is strongly measurable if and only if $uy$ is strongly measurable, and $\|(uy)(\cdot)\| = u\|y(\cdot)\|$. Hence $y \in E_u(X)$ if and only if $Jy = uy \in E(X)$, and $\|Jy\|_{E(X)} = \|y\|_{E_u(X)}$. Thus $J$ is a surjective linear isometry and $\|\tilde{T}\| = \|T\|$. Clearly $J$ commutes with the operators of multiplication by functions from $L^\infty$.

Let $G \in \Sigma$ be arbitrary, and let $x \in E(X)$. For every $n \in \mathds{N}$ we have $\mu(G \cap \Omega_n) < \infty$ and therefore
\begin{equation*}
\chi_{\Omega_n}(\chi_G Tx) = P_{G \cap \Omega_n}Tx = TP_{G \cap \Omega_n}x = TP_{\Omega_n}(\chi_G x) = \chi_{\Omega_n}T(\chi_G x).
\end{equation*}
Since $\bigcup_n \Omega_n = \Omega$, we obtain $\chi_G Tx = T(\chi_G x)$, i.e., $TP_G = P_GT$ for every $G \in \Sigma$. Because $J$ commutes with $P_G$, so does $\tilde{T} = J^{-1}TJ$. The same argument shows that if $T$ commutes with $L^\infty$ then so does $\tilde{T}$.

Finally, let $A$ be as in the last statement of the lemma, and let $x \in E(X)$. Put $y = J^{-1}x = x/u \in E_u(X)$. Then for almost all $\omega$
\begin{equation*}
Tx(\omega) = (J\tilde{T}y)(\omega) = u(\omega)A(\omega)y(\omega) = A(\omega)\big(u(\omega)y(\omega)\big) = A(\omega)x(\omega). \qedhere
\end{equation*}
\end{proof}

\begin{theorem} \label{t1}
 Let $(\Omega, \Sigma, \mu)$ be a complete measure space with a $\sigma$-finite measure $\mu$, and let $E$ be a Banach function space such that $\chi_G \in E$ for every $G \in \Sigma$ with $\mu(G) < \infty$.
  Let $X$ be a separable Banach space, and let $E(X)$ be the corresponding K\"othe--Bochner space. For $G \in \Sigma$, $\mu(G) < \infty$, let $P_G$ be the projection on $E(X)$ corresponding to the operator of multiplication by $\chi_G$.
 Let $T$ be a bounded linear operator on $E(X)$ such that
  \begin{equation}\label{eq1}
    TP_G = P_GT, \quad G \in \Sigma, \ \mu(G) < \infty.
 \end{equation}
   Then there are a set $\bar{\Omega} \in \Sigma$ with $\mu(\Omega \setminus \bar{\Omega}) = 0$ and a function $A: \bar{\Omega} \rightarrow L(X)$ such that $\|A(\omega)\| \leq \|T\|$ for all $\omega \in \bar{\Omega}$ and, for every $x \in E(X)$,
   \begin{equation*}
     Tx(\omega) = A(\omega)x(\omega) \quad \text{for almost all } \omega \in \Omega.
   \end{equation*}
   In particular, $\sup_{\omega \in \bar{\Omega}} \|A(\omega)\| = \|T\|$.
\end{theorem}
\begin{proof}
By Lemma~\ref{l0} we can replace $\mu$, $E$ and $T$ by $\nu$, $E_u$ and $\tilde{T}$, respectively: the hypotheses of the theorem are preserved, $\|\tilde{T}\| = \|T\|$, and a function $A$ representing $\tilde{T}$ represents $T$ as well. Therefore we can assume without loss of generality that $(\Omega, \Sigma, \mu)$ is a complete probability space and that $\chi_\Omega \in E$.

  Let $Y = \{x_i: i \in \mathds{N}\}$ be a dense $\mathds{Q}(i)$-linear subspace of $X$. For every $i \in \mathds{N}$ we define $\vec{x}_i$ by $\vec{x}_i(\omega) \equiv x_i$, $\omega \in \Omega$. Notice that because
  $\chi_\Omega \in E$ we have $\vec{x}_i \in E(X)$. We fix representatives of the classes $T\vec{x}_i \in E(X)$, $i \in \mathds{N}$.

For every $i \in \mathds{N}$ the equality
  \begin{equation} \label{eq2}
  A(\omega)x_i = T\vec{x}_i(\omega)
  \end{equation}
  defines $A(\omega)x_i$ for every $\omega \in \Omega$. Because the operator $T$ is linear and $Y$ is countable, the equality~(\ref{eq2}) defines a $\mathds{Q}(i)$-linear operator $A(\omega): Y \rightarrow X$ for all $\omega$ from a set $\Omega_1 \subset \Omega$ such that $\mu(\Omega \setminus \Omega_1) = 0$.

  Let us fix an $i \in \mathds{N}$. We claim that there is a subset $N_i$ of $\Omega$ such that $\mu(N_i) =0$ and
      \begin{equation}\label{eq3}
\|A(\omega)x_i\| \leq \|T\|\|x_i\|, \quad \omega \in \Omega_1 \setminus N_i.
      \end{equation}
       Indeed, assume that~(\ref{eq3}) is not true. By~(\ref{eq2}) the function $\omega \mapsto \|A(\omega)x_i\|$ is measurable, and therefore there are an $F \in \Sigma$, $F \subset \Omega_1$, and an $\varepsilon >0$ such that $0 < \mu(F) < \infty$ and $\|A(\omega)x_i\| \geq (\|T\|+\varepsilon)\|x_i\|$ for all $\omega \in F$. Let
  $u_i = \chi_F \vec{x}_i \in E(X)$. Because $Tu_i = T\chi_F \vec{x}_i = \chi_F T\vec{x}_i$, the definition of the norm in $E(X)$ gives $\|Tu_i\| \geq (\|T\| + \varepsilon)\|u_i\|$, a contradiction.

  Let $\bar{\Omega} = \Omega_1 \setminus \bigcup \limits_{i=1}^\infty N_i$. Then $\mu(\bar{\Omega}) = 1$ and for every $\omega \in \bar{\Omega}$ the equality~(\ref{eq2}) defines a $\mathds{Q}(i)$-linear operator $A(\omega) : Y \rightarrow X$ such that $\|A(\omega)\| \leq \|T\|$.

  Therefore for every $\omega \in \bar{\Omega}$ the operator $A(\omega)$ can be extended uniquely to a bounded linear operator from $X$ to $X$ (the extension is additive and $\mathds{Q}(i)$-homogeneous, hence $\mathds{C}$-homogeneous by continuity). We still denote this extension by $A(\omega)$. Notice that $\|A(\omega)\| \leq \|T\|$ for $\omega \in \bar{\Omega}$.

  It is well known (see e.g.~\cite[Lemma 3.2]{DKO}; the proof there works without changes for finite or countable families of sets $G_i$) that the set of $\sigma$-valued functions is dense in $E(X)$. We consider a subset $Z$ of the set of all $\sigma$-valued functions in $E(X)$. Namely, we will say that $z \in Z$ if we can choose the vectors $z_i$ and the sets $G_i$ in Definition~\ref{d1} in such a way that $z_i \in Y$ (below we denote these vectors by $y_i$), $G_i \subset \bar{\Omega}$, and
  $\bigcup \limits_{i \in I} G_i = \bar{\Omega}$.

  Since $Y$ is dense in $X$ and $\chi_\Omega \in E$, every $\sigma$-valued function can be approximated in $E(X)$ by elements of $Z$, and therefore $Z$ is dense in $E(X)$. We claim that for any $z \in Z$ we have
  $Tz(\omega) = A(\omega)z(\omega)$ for almost all $\omega \in \Omega$.
  Indeed, let $z \in Z$ and let $G_i \in \Sigma$ be such that $G_i \subset \bar{\Omega}$,
  $\mu(G_i) >0$, and $\chi_{G_i}z = \chi_{G_i}\vec{y}_i$, where $y_i \in Y$.
  Then
  $\chi_{G_i}Tz = T\chi_{G_i}z = T\chi_{G_i}\vec{y}_i = \chi_{G_i}T\vec{y}_i$. Therefore, for almost all $\omega \in G_i$ we have $Tz(\omega) = T\vec{y}_i(\omega) = A(\omega)y_i = A(\omega)z(\omega)$. Because
  $\bigcup \limits_{i \in I} G_i = \bar{\Omega}$ and $I$ is countable, the claim follows.

  Let us now consider an $x \in E(X)$ and a sequence $z_n \in Z$ converging to $x$ in $E(X)$. Then $Tz_n \to Tx$ in $E(X)$, and therefore the scalar functions $\|z_n(\cdot) - x(\cdot)\|$ and $\|Tz_n(\cdot) - Tx(\cdot)\|$ converge to $0$ in $E$, hence in measure on sets of finite measure (see e.g.~\cite[Theorem~1, p.~96]{KA}). Passing to a subsequence we may assume that $z_n(\omega) \to x(\omega)$ and $Tz_n(\omega) \to Tx(\omega)$ in norm for almost all $\omega$. Fix such an $\omega \in \bar{\Omega}$ for which, in addition, $Tz_n(\omega) = A(\omega)z_n(\omega)$ for all $n$. Since $A(\omega) \in L(X)$, we have $A(\omega)z_n(\omega) \to A(\omega)x(\omega)$, and therefore $A(\omega)x(\omega) = Tx(\omega)$. Thus $Tx(\omega) = A(\omega)x(\omega)$ for almost all $\omega \in \Omega$. In particular, the function $\omega \mapsto A(\omega)x(\omega)$ is almost everywhere the limit of the strongly measurable functions $Tz_n$, hence it is strongly measurable, and the inequality $\|A(\omega)\| \leq \|T\|$ shows that it belongs to $E(X)$.

  Finally, the inequality $\|T\| \leq \sup_{\omega \in \bar{\Omega}} \|A(\omega)\|$ follows from $\|Tx(\omega)\| \leq \|A(\omega)\|\,\|x(\omega)\|$ and the definition of the norm in $E(X)$.
    \end{proof}

\begin{remark} \label{rnotnorm}
The function $A$ in Theorem~\ref{t1} is measurable in the strong operator topology, i.e., the functions $\omega \mapsto A(\omega)x$, $x \in X$, are measurable, but in general it is not measurable as an $L(X)$-valued function with the operator norm, even if $T$ is very simple. Let $\Omega = [0,1]$ with Lebesgue measure, let $E$ be an arbitrary Banach function space on $[0,1]$ as in Theorem~\ref{t1}, let $X = L^2[0,1]$, and let $b(t) \in L(X)$ be the operator of multiplication by $\chi_{[0,t]}$. Since $\|b(t)y - b(s)y\| = \|\chi_{(s,t]}y\|_2 \to 0$ as $s \to t$, the map $t \mapsto b(t)$ is continuous in the strong operator topology. Hence for every $x \in E(X)$ the function $t \mapsto b(t)x(t)$ is strongly measurable (it is the limit of $b(t)s_n(t)$ for countably valued $s_n \to x$), and $(Tx)(t) = b(t)x(t)$ defines an operator $T$ on $E(X)$ with $\|T\| \leq 1$ that commutes with $L^\infty$. If $A$ is any function as in Theorem~\ref{t1}, then $A(t)x_i = (T\vec{x}_i)(t) = b(t)x_i$ for a dense sequence $(x_i)$ in $X$ and almost all $t$, so that $A(t) = b(t)$ for almost all $t$. On the other hand, $\|b(t) - b(s)\| = 1$ for $s \neq t$, so for every set $N$ of measure zero the set $\{b(t) : t \notin N\}$ is not separable. Therefore no function $A$ representing $T$ is the almost everywhere limit, in the operator norm, of countably valued measurable functions. This is the reason why in Section~\ref{s4} we have to assume norm-measurability of the weight explicitly (see Remark~\ref{r30}).
\end{remark}

    \section{The case when $X$ is complemented in $X^{\star \star}$}

    The arguments used in the proof of Theorem~\ref{t1} fail in the case when the Banach space $X$ is not separable. Nevertheless, this difficulty can be partially circumvented by employing a lifting of $L^\infty$.

    Let $(\Omega, \Sigma, \mu)$ be a complete measure space with a $\sigma$-finite measure.
    Let $\mathcal{L}^\infty(\Omega, \Sigma, \mu)$ be the Banach space of all bounded $\mu$-measurable functions on $\Omega$ with the supremum norm.  When it does not cause any ambiguity we will use the notation $\mathcal{L}^\infty$. Then (see~\cite{vN31, Ma}) there is a map $\Lambda: L^\infty \rightarrow \mathcal{L}^\infty$ with the following properties:
    \begin{enumerate}[(a)]
      \item $\Lambda(f) = f$ almost everywhere;
      \item $\Lambda(\alpha f + \beta g) = \alpha \Lambda(f) + \beta \Lambda(g)$, $\alpha, \beta \in \mathds{C}$;
      \item $\Lambda(fg) = \Lambda(f)\Lambda(g)$;
      \item $\|\Lambda(f)\| = \|f\|$.
    \end{enumerate}
A map $\Lambda$ with properties (a)--(d) is called a lifting of $L^\infty$.

The following lemma is a special case of a known property of liftings (see~\cite[Chapter VI, \S5, Proposition 2, p.~78]{IT}, applied, in the notation of \cite{IT}, with $E = X^\star$ and $E' = X^{\star\star}$). We include the proof for the reader's convenience.

\begin{lemma} \label{l1}
  Let $(\Omega, \Sigma, \mu)$ be a complete measure space with a $\sigma$-finite measure $\mu$, let $\Lambda$ be a lifting of $L^\infty$, and let $f : \Omega \rightarrow X$ be a strongly measurable function such that $\|f(\omega)\| \leq C$ for almost all $\omega \in \Omega$. Then for every $\omega \in \Omega$ there is a unique $\hat{f}(\omega) \in X^{\star\star}$ such that
  \begin{equation*}
    \langle \hat{f}(\omega), x^\star \rangle = \Lambda(\langle f, x^\star \rangle)(\omega), \quad x^\star \in X^\star,
  \end{equation*}
  and there is a set $N \in \Sigma$, $\mu(N) = 0$, such that $\hat{f}(\omega) = f(\omega)$ for all $\omega \in \Omega \setminus N$.
\end{lemma}

\begin{proof}
  For every $\omega \in \Omega$ the map $x^\star \mapsto \Lambda(\langle f, x^\star \rangle)(\omega)$ is linear by (b), and by (d) we have $|\Lambda(\langle f, x^\star \rangle)(\omega)| \leq \|\langle f, x^\star \rangle\|_\infty \leq C\|x^\star\|$. This proves the existence and uniqueness of $\hat{f}(\omega)$.

  Because $f$ is strongly measurable there are a set $N_0 \in \Sigma$, $\mu(N_0) = 0$, and a closed separable subspace $X_0$ of $X$ such that $f(\omega) \in X_0$ for all $\omega \in \Omega \setminus N_0$. Fix $\varepsilon > 0$ and let $\{d_k : k \in \mathds{N}\}$ be a dense subset of $X_0$. Let
  $H_k = \{\omega \in \Omega \setminus N_0 : \|f(\omega) - d_k\| < \varepsilon\}$ and $G_k = H_k \setminus \bigcup \limits_{j < k} H_j$. Then the sets $G_k \in \Sigma$ are pairwise disjoint, $\bigcup \limits_{k=1}^\infty G_k = \Omega \setminus N_0$, and
  \begin{equation} \label{eq5}
    \|f(\omega) - d_k\| < \varepsilon, \quad \omega \in G_k.
  \end{equation}

  By (c) we have $\Lambda(\chi_{G_k})^2 = \Lambda(\chi_{G_k})$, therefore $\Lambda(\chi_{G_k}) = \chi_{G_k^\prime}$ for some $G_k^\prime \in \Sigma$, and by (a), $\mu(G_k \triangle G_k^\prime) = 0$. Let
  $N_\varepsilon = N_0 \cup \bigcup \limits_{k=1}^\infty (G_k \triangle G_k^\prime)$. Then $\mu(N_\varepsilon) = 0$ and every $\omega \in \Omega \setminus N_\varepsilon$ belongs to $G_k \cap G_k^\prime$ for some $k \in \mathds{N}$. Notice that the set $N_\varepsilon$ does not depend on $x^\star$.

  Fix $k \in \mathds{N}$, $\omega \in G_k \cap G_k^\prime$, and $x^\star \in X^\star$. Because $\Lambda(\chi_{G_k})(\omega) = 1$, properties (b) and (c) imply
  \begin{equation*}
    \begin{split}
    \langle \hat{f}(\omega), x^\star \rangle & = \Lambda(\chi_{G_k})(\omega)\,\Lambda(\langle f, x^\star \rangle)(\omega) = \Lambda(\chi_{G_k}\langle f, x^\star \rangle)(\omega) \\
    & = \langle d_k, x^\star \rangle \Lambda(\chi_{G_k})(\omega) + \Lambda(\chi_{G_k}\langle f - d_k, x^\star \rangle)(\omega) \\
    & = \langle d_k, x^\star \rangle + \Lambda(\chi_{G_k}\langle f - d_k, x^\star \rangle)(\omega).
    \end{split}
  \end{equation*}
  By~(\ref{eq5}), $\|\chi_{G_k}\langle f - d_k, x^\star \rangle\|_\infty \leq \varepsilon \|x^\star\|$, and therefore, by (d), $|\langle \hat{f}(\omega) - d_k, x^\star \rangle| \leq \varepsilon\|x^\star\|$. Because $x^\star \in X^\star$ is arbitrary, $\|\hat{f}(\omega) - d_k\|_{X^{\star\star}} \leq \varepsilon$. Combining this with~(\ref{eq5}) we see that
  \begin{equation*}
    \|\hat{f}(\omega) - f(\omega)\| < 2\varepsilon, \quad \omega \in \Omega \setminus N_\varepsilon.
  \end{equation*}
  It remains to put $N = \bigcup \limits_{n=1}^\infty N_{1/n}$.
\end{proof}

\begin{theorem} \label{t2}
  Let $(\Omega, \Sigma, \mu)$ and $E$ be as in Theorem~\ref{t1}, and let $E(X)$ be the corresponding K\"othe--Bochner space. Assume that the Banach space $X$ is complemented in its second dual $X^{\star\star}$, and let $Q$ be a bounded projection of $X^{\star\star}$ onto $X$. Let $T$ be a bounded linear operator on $E(X)$ that commutes with $L^\infty$.

  Then there is a map $A: \Omega \rightarrow L(X)$ such that
  \begin{equation*}
    \|T\| \leq \sup \limits_{\omega \in \Omega} \|A(\omega)\| \leq \|Q\| \, \|T\|,
  \end{equation*}
  and for every $x \in E(X)$ we have $Tx(\omega) =A(\omega)x(\omega)$ a.e. In particular, if $X$ is $1$-complemented in $X^{\star\star}$, then
  $\sup \limits_{\omega \in \Omega} \|A(\omega)\| = \|T\|$.
\end{theorem}

\begin{proof}
  By Lemma~\ref{l0}, as in the proof of Theorem~\ref{t1}, we can assume without loss of generality that $\mu$ is a probability measure and $\chi_\Omega \in E$; this does not change $L^\infty$, the norm of $T$, or the functions $A$ representing $T$. Let $x \in X$ and $x^\star \in X^\star$. As in the proof of Theorem~\ref{t1} we consider $\vec{x} \in E(X)$, $\vec{x}(\omega) = x$, $\omega \in \Omega$. The equality
  $a_{x,x^\star}(\omega) = \langle T\vec{x}(\omega), x^\star\rangle$ defines $a_{x,x^\star}(\omega)$ a.e. on $\Omega$. The argument used to prove~(\ref{eq3}) shows that
  \begin{equation}\label{eq4}
     |a_{x,x^\star}(\omega)| \leq \|(T\vec{x})(\omega)\| \, \|x^\star\|
   \leq \|T\| \,\|x\| \, \|x^\star\|
  \end{equation}
  for almost all $\omega \in \Omega$. Thus $a_{x,x^\star} \in L^\infty$.

  Let us fix a lifting $\Lambda: L^\infty \rightarrow \mathcal{L}^\infty$ and let $b_{x,x^\star} = \Lambda(a_{x,x^\star})$. It follows from~(\ref{eq4}) and properties (b) and (d) of the lifting $\Lambda$ that for every
  $\omega \in \Omega$ the map $(x,x^\star) \mapsto b_{x,x^\star}(\omega)$ is a bounded bilinear form on $X \times X^\star$. Therefore (see~\cite[p.~24]{Ry}) for every $\omega \in \Omega$ there is a bounded linear operator $B(\omega): X \rightarrow X^{\star\star}$ such that $b_{x,x^\star}(\omega) = \langle B(\omega)x,x^\star\rangle$, $x \in X$, $x^\star \in X^\star$. Notice that $\|B(\omega)\| \leq \|T\|$ for every $\omega \in \Omega$.

  By (\ref{eq4}) the function $T\vec{x}$ is essentially bounded, and by Lemma~\ref{l1}, $T\vec{x}(\omega) = B(\omega)x$ for almost all $\omega \in \Omega$.

  Let $z$ be a $\sigma$-valued function, $\chi_{G_i}z = \chi_{G_i}\vec{z}_i$ (Definition~\ref{d1}). Then $\chi_{G_i}Tz = T\chi_{G_i}\vec{z}_i = \chi_{G_i}T\vec{z}_i$, and therefore $Tz(\omega) = B(\omega)z_i = B(\omega)z(\omega)$ for almost all $\omega \in G_i$. Since there are at most countably many sets $G_i$, $Tz(\omega) = B(\omega)z(\omega)$ almost everywhere. Now let $x \in E(X)$. By \cite[Lemma 3.2]{DKO} there are $\sigma$-valued functions $z_n$ such that $z_n \to x$ in $E(X)$. As in the proof of Theorem~\ref{t1}, passing to a subsequence we may assume that $z_n(\omega) \to x(\omega)$ and $Tz_n(\omega) \to Tx(\omega)$ in norm for almost all $\omega$. Since $\|B(\omega)\| \leq \|T\|$ for \emph{every} $\omega$, we have $B(\omega)z_n(\omega) \to B(\omega)x(\omega)$, and hence $Tx(\omega) = B(\omega)x(\omega)$ for almost all $\omega$. Note that no separability of $X$ is needed here, because the convergence is in norm and the operators $B(\omega)$ are uniformly bounded.

  Put $A(\omega) = QB(\omega)$, $\omega \in \Omega$. Since $Tx(\omega) \in X$, we have $A(\omega)x(\omega) = QTx(\omega) = Tx(\omega)$ for almost all $\omega$, and $\|A(\omega)\| \leq \|Q\|\,\|T\|$. Finally, $\|T\| \leq \sup_\omega\|A(\omega)\|$ follows from the definition of the norm in $E(X)$.
\end{proof}

Every dual Banach space is complemented in its second dual by a projection
of norm one. Therefore Theorem~\ref{t2} applies, with the sharp norm
estimate, whenever $X$ is a dual space. We state this special case
explicitly.

\begin{corollary} \label{c4}
  Let $(\Omega, \Sigma, \mu)$ and $E$ be as in Theorem~\ref{t1}, let
  $X = Y^\star$ for some Banach space $Y$, and let $T$ be a bounded linear
  operator on $E(X)$ that commutes with $L^\infty$. Then there is a map
  $A : \Omega \rightarrow L(X)$ such that
  $\sup \limits_{\omega \in \Omega} \|A(\omega)\| = \|T\|$ and for every
  $x \in E(X)$ we have $Tx(\omega) = A(\omega)x(\omega)$ for almost all
  $\omega \in \Omega$. In particular, this holds for every reflexive
  Banach space $X$.
\end{corollary}

\begin{proof}
  Let $j_Y : Y \rightarrow Y^{\star\star}$ be the canonical embedding.
  Then $Q = j_Y^\star : Y^{\star\star\star} \rightarrow Y^\star$ is a
  projection of $X^{\star\star}$ onto $X$ and $\|Q\| = 1$. It remains to
  apply Theorem~\ref{t2}. If $X$ is reflexive then $X = (X^\star)^\star$.
\end{proof}

\begin{remark} \label{r1}
  The proof of Theorem~\ref{t2} shows that if $E(X)$ is an arbitrary K\"othe--Bochner space and $T$ is a bounded linear operator on $E(X)$ that commutes with $L^\infty$ then there is a map $A: \Omega \rightarrow L(X,X^{\star\star})$ such that $\sup \limits_{\omega \in \Omega} \|A(\omega)\| = \|T\|$ and for any $x \in E(X)$ we have $Tx(\omega) = A(\omega)x(\omega)$ a.e.
\end{remark}

\begin{remark} \label{r2}
The restriction to $L(X,X^{\star\star})$ in Remark~\ref{r1} has no counterpart in the topological setting. Let $K$ be a compact Hausdorff space and let $T$ be a bounded linear operator on $C(K,X)$ that commutes with the operators of multiplication by functions from $C(K)$. Then there is a strongly continuous map $A : K \to L(X)$ such that $Tf(k) = A(k)f(k)$, $f \in C(K,X)$, $k \in K$, and $\sup_{k \in K} \|A(k)\| = \|T\|$; conversely, every such map defines an operator commuting with $C(K)$ (see \cite[Proposition~4.7(ii), p.~84]{Be}, applied to the function module $C(K,X)$, \cite[Example~1, pp.~77--78]{Be}). The same reference \cite[Example~2, p.~78]{Be} covers $C_0(L,X)$ for a locally compact Hausdorff space $L$; in that case, in the converse statement, the strongly continuous map $A$ has to be assumed norm bounded, since the uniform boundedness principle guarantees this only when $L$ is compact. (For the centralizer of $C_0(L,X)$ see \cite[Proposition~10.3, p.~169]{Be}.) Note also that the supremum $\sup_k \|A(k)\|$ need not be attained, because $k \mapsto \|A(k)\|$ is only lower semicontinuous.

The reason why $A(k) \in L(X)$ automatically is that in $C(K,X)$ one can evaluate at every point: if $f(k_0) = 0$ then $(Tf)(k_0) = 0$, and therefore $A(k)x = (T\vec{x})(k)$ is a vector of $X$. If $K$ is the Stone space of $L^\infty(\Omega, \Sigma, \mu)$, then $C(K,X)$ can be identified only with the subspace of $L^\infty(\Omega, \Sigma, \mu; X)$ consisting of the functions whose essential range is relatively compact. For example, if $\{G_i\}$ are pairwise disjoint sets of positive measure and $\{e_i\}$ is an orthonormal sequence in $\ell^2$, then $\sum_i \chi_{G_i} e_i$ belongs to $L^\infty(\Omega, \Sigma, \mu; \ell^2)$ but not to $C(K, \ell^2)$. In $E(X)$ only equivalence classes of functions are available, and the lifting used in the proof of Theorem~\ref{t2} produces, in general, only operators $A(\omega) : X \to X^{\star\star}$.
\end{remark}

\begin{remark} \label{ratomic}
If the measure $\mu$ is purely atomic, then the conclusions of Theorems~\ref{t1} and~\ref{t2} hold for an arbitrary Banach space $X$ and with $\sup_\omega \|A(\omega)\| = \|T\|$. Indeed, every element of $E(X)$ is constant on each atom, so for an atom $G$ of finite measure the band $\chi_G E(X)$ is isomorphic to $X$. If $T$ commutes with $L^\infty$, it leaves $\chi_G E(X)$ invariant and acts there as an operator $A_G \in L(X)$ with $\|A_G\| \leq \|T\|$. Consequently, in this case Corollaries~\ref{c1} and~\ref{c1a} and Theorem~\ref{t4} also hold for every Banach space $X$ (in Theorem~\ref{t4} for every $1 \leq p < \infty$ and without the assumption that the weight is norm-measurable), and the operators $p(\omega)$ are projections for every $\omega$. Thus the difficulties considered in this paper come entirely from the nonatomic part of $\mu$.
\end{remark}

\section{Applications to weighted composition operators on K\"othe--Bochner spaces. Spectral projections} \label{s3}

  \begin{definition} \label{dergodic}
    Let $(\Omega, \Sigma, \mu)$ be a measure space, and let $\varphi : \Omega \rightarrow \Omega$ be a measurable map of $\Omega$ onto itself. The measure $\mu$ is called ergodic with respect to $\varphi$ if it follows from
  $G \in \Sigma$ and  $\mu(\varphi^{-1}(G) \triangle G)=0$ (where $\triangle$ denotes the symmetric difference) that either $\mu(\Omega \setminus G) =0$ or $\mu(G) = 0$.
  \end{definition}

  \begin{definition} \label{daperiodic}
  Let $(\Omega, \Sigma, \mu)$ be a complete measure space with a $\sigma$-finite measure $\mu$, and let $\varphi : \Omega \to \Omega$ be an invertible map such that $\varphi$ and $\varphi^{-1}$ are measurable and \emph{nonsingular}, i.e., $\mu(\varphi(G)) = 0$ if and only if $\mu(G) = 0$, $G \in \Sigma$. A set $F \in \Sigma$ is called \emph{$N$-separated} if $\mu(F \cap \varphi^k(F)) = 0$ for $1 \leq k \leq N$; then also $\mu(\varphi^i(F) \cap \varphi^j(F)) = 0$ for $i \neq j$, $|i - j| \leq N$. The map $\varphi$ is called \emph{aperiodic} if for every $N \in \mathds{N}$ every set of positive measure contains an $N$-separated subset of positive measure.
  \end{definition}

  Since a subset of an $N$-separated set is $N$-separated and $\mu$ is $\sigma$-finite, the $N$-separated subset can be chosen of finite positive measure. If $\varphi$ is aperiodic, then every measurable subset of $\{\omega : \varphi^k(\omega) = \omega\}$, $k \neq 0$, has measure zero, because every such subset $F$ satisfies $\varphi^k(F) = F$. The converse is not true in general (see Remark~\ref{rperiodic}).

  \begin{remark} \label{rfrolik}
  Let $\varphi$ be as in Definition~\ref{daperiodic}. Then $f \mapsto f \circ \varphi$ is an automorphism of the algebra $L^\infty(\mu)$. Let $\mathcal{K}$ be the Gelfand compact of $L^\infty(\mu)$, so that $L^\infty(\mu) \cong C(\mathcal{K})$, and let $\psi$ be the unique homeomorphism of $\mathcal{K}$ onto itself such that $\widehat{f \circ \varphi} = \hat{f} \circ \psi$, where $\hat{f}$ is the Gelfand transform of $f$. Clopen subsets of $\mathcal{K}$ correspond to sets $G \in \Sigma$ (up to sets of measure zero) via $\hat{\chi}_G = \chi_{\hat{G}}$, and $\hat{G} \neq \emptyset$ if and only if $\mu(G) > 0$. The following conditions are equivalent:
  \begin{enumerate}[(i)]
  \item $\varphi$ is aperiodic;
  \item $\psi$ has no periodic points;
  \item there are no $k \in \mathds{N}$ and $G \in \Sigma$, $\mu(G) > 0$, such that $\mu(H \triangle \varphi^k(H)) = 0$ for every $H \in \Sigma$, $H \subset G$.
  \end{enumerate}
  Indeed, if $G$ is as in the negation of (iii), then every $H \subset G$ of positive measure satisfies $\mu(H \cap \varphi^k(H)) = \mu(H) > 0$, so $\varphi$ is not aperiodic; thus (i) implies (iii). If $\psi^k(\kappa) = \kappa$ for some $\kappa \in \mathcal{K}$, then by Frol\'{\i}k's theorem \cite{Fr} (since $\mu$ is $\sigma$-finite, $L^\infty(\mu)$ is a complete lattice and $\mathcal{K}$ is extremally disconnected) the set of fixed points of $\psi^k$ is a nonempty clopen set $\hat{G}$. Then $\chi_G(f \circ \varphi^k) = \chi_Gf$ for all $f \in L^\infty(\mu)$, and since $\psi^k(\hat{G}) = \hat{G}$, also $\varphi^k(G) = G$ up to a set of measure zero. For $f = \chi_H$, $H \subset G$, we get $\mu(H \triangle \varphi^{-k}(H)) = 0$, which contradicts (iii); thus (iii) implies (ii). Finally, assume (ii), let $N \in \mathds{N}$ and $\mu(G) > 0$, and choose $\kappa \in \hat{G}$. The points $\kappa, \psi(\kappa), \dots, \psi^N(\kappa)$ are pairwise distinct, and since the clopen sets form a base of the topology of $\mathcal{K}$, there is a nonempty clopen set $\hat{F} \subset \hat{G}$ containing $\kappa$ such that the sets $\hat{F}, \psi(\hat{F}), \dots, \psi^N(\hat{F})$ are pairwise disjoint. Then $F \subset G$ is an $N$-separated set of positive measure; thus (ii) implies (i).
  \end{remark}

  We say that a set $\sigma'$ is a \emph{clopen subset} of $\sigma(T)$ if it is both closed and open in $\sigma(T)$; these are exactly the subsets of $\sigma(T)$ that have a spectral (Riesz) projection.

 \begin{corollary} \label{c1}
     Let $(\Omega, \Sigma, \mu)$ be a complete measure space with a $\sigma$-finite measure $\mu$. Let $E$ be as in Theorem~\ref{t1}, let $X$ be a separable Banach space, and let $E(X)$ be the corresponding K\"othe--Bochner space. Let $\varphi : \Omega \rightarrow \Omega$ be an invertible measurable nonsingular map, and assume that $\varphi$ is aperiodic (Definition~\ref{daperiodic}).

      Let $a: \Omega \rightarrow L(X)$ be such that the formula
      $Tx(\omega) = a(\omega)x(\varphi(\omega))$, $x \in E(X)$, $\omega \in \Omega$, defines a bounded invertible operator on $E(X)$.

     If $\sigma'$ is a nonempty clopen subset of $\sigma(T)$ (for example, a connected component of $\sigma(T)$ that is open in $\sigma(T)$) and $P$ is the corresponding spectral projection, then there are an $\bar{\Omega} \in \Sigma$ and $p: \bar{\Omega} \rightarrow L(X)$ such that $\mu(\Omega \setminus \bar{\Omega}) =0$, $p(\omega)$ is a bounded projection on $X$ for every $\omega \in \bar{\Omega}$, $\sup \limits_{\omega \in \bar{\Omega}} \|p(\omega)\| = \|P\|$, and, for every $x \in E(X)$, $Px(\omega) = p(\omega)x(\omega)$ for almost all $\omega \in \bar{\Omega}$.
    \end{corollary}

    \begin{proof}
      Since $\varphi$ is nonsingular, $f \circ \varphi^n \in L^\infty(\mu)$ for $f \in L^\infty(\mu)$, and $Tf = (f \circ \varphi)T$; hence $T^nf = (f \circ \varphi^n)T^n$, $n \in \mathds{Z}$.

      \emph{Step 1: circles.} Let $r > 0$ be such that the circle $\{|\lambda| = r\}$ does not meet $\sigma(T)$, let $\Pi_r$ be the spectral projection of $T$ corresponding to $\sigma(T) \cap \{|\lambda| < r\}$, and let $E_1 = \Pi_rE(X)$, $E_2 = (I - \Pi_r)E(X)$. We claim that
      \begin{align*} E_1 &= \{x \in E(X) : r^{-n}\|T^nx\| \to 0\}, \\ E_2 &= \{x \in E(X) : r^{n}\|T^{-n}x\| \to 0\}. \end{align*}
      Indeed, the spectral radius of $T|_{E_1}$ is less than $r$, so $r^{-n}\|T^nx\| \to 0$ for $x \in E_1$. Conversely, $T|_{E_2}$ is invertible and the spectral radius of its inverse is less than $1/r$, so there are $C > 0$ and $s > r$ such that $\|T^ny\| \geq Cs^n\|y\|$, $y \in E_2$, $n \in \mathds{N}$. If $r^{-n}\|T^nx\| \to 0$ and $y = (I - \Pi_r)x$, then $Cs^n\|y\| \leq \|T^ny\| = \|(I - \Pi_r)T^nx\| \leq \|I - \Pi_r\|\,\|T^nx\|$, hence $(s/r)^n\|y\| \to 0$ and $y = 0$. The second equality is proved in the same way. Now let $f \in L^\infty(\mu)$. Since $\|T^nfx\| = \|(f \circ \varphi^n)T^nx\| \leq \|f\|_\infty\|T^nx\|$ for $n \in \mathds{Z}$, we have $fE_1 \subset E_1$ and $fE_2 \subset E_2$. Therefore $\Pi_rf\Pi_r = f\Pi_r$ and $(I-\Pi_r)f(I-\Pi_r) = f(I-\Pi_r)$, so that $\Pi_rf = f\Pi_r$.

      \emph{Step 2: arbitrary clopen subsets.} We consider $E(X)$ as a Banach $C(\mathcal{K})$-module, where $\mathcal{K}$ is the Gelfand compact of $L^\infty(\mu)$. Since $Tf = (f \circ \varphi)T$, $T$ is a weighted composition operator on this module with respect to the homeomorphism $\psi$ of Remark~\ref{rfrolik}. Since $\varphi$ is aperiodic, $\psi$ has no periodic points (Remark~\ref{rfrolik}), and by~\cite[Corollary 12.30]{AAK} the spectrum $\sigma(T)$ is rotation invariant. Since $T$ is invertible, the set $R = \{|\lambda| : \lambda \in \sigma(T)\}$ is a compact subset of $(0, \infty)$. Every circle $\{|\lambda| = t\}$, $t \in R$, is a connected subset of $\sigma(T)$, and therefore it is contained either in $\sigma'$ or in $\sigma(T) \setminus \sigma'$. Thus $\sigma' = \{\lambda \in \sigma(T) : |\lambda| \in J\}$, where $J$ and $R \setminus J$ are disjoint compact sets. Consequently there are $0 < s_1 < t_1 < s_2 < \dots < t_k$ such that $s_i, t_i \notin R$ and $J = R \cap \bigcup_{i=1}^k [s_i, t_i]$, and
      \[ P = \sum_{i=1}^k (\Pi_{t_i} - \Pi_{s_i}). \]
      By Step 1, $Pf = fP$ for every $f \in L^\infty(\mu)$, and by Theorem~\ref{t1} there is $p : \bar{\Omega} \to L(X)$ with $\|p(\omega)\| \leq \|P\|$ such that $Px(\omega) = p(\omega)x(\omega)$ a.e. for every $x \in E(X)$; clearly also $\|P\| \leq \sup_\omega \|p(\omega)\|$. Applying $P^2 = P$ to the functions $\chi_G\vec{x}_i$, where $\{x_i\}$ is a countable dense subset of $X$ and $\mu(G) < \infty$, and removing from $\bar{\Omega}$ a set of measure zero, we obtain $p(\omega)^2 = p(\omega)$ for all $\omega \in \bar{\Omega}$.
    \end{proof}

    \begin{corollary} \label{c1a}
      Assume the conditions of Corollary~\ref{c1}, except that instead of separability we assume that $X$ is complemented in $X^{\star\star}$, and let $Q$ be a bounded projection of $X^{\star\star}$ onto $X$. If $\sigma'$ is a nonempty clopen subset of $\sigma(T)$ and $P$ is the corresponding spectral projection, then there is a map $p : \Omega \rightarrow L(X)$ such that
      \begin{enumerate}[(a)]
        \item $\|P\| \leq \sup \limits_{\omega \in \Omega} \|p(\omega)\| \leq \|Q\|\,\|P\|$;
        \item for every $x \in E(X)$ we have $Px(\omega) = p(\omega)x(\omega)$ for almost all $\omega \in \Omega$;
        \item for every $x \in E(X)$ we have $p(\omega)^2x(\omega) = p(\omega)x(\omega)$ for almost all $\omega \in \Omega$; in particular, for every $y \in X$ we have $p(\omega)^2y = p(\omega)y$ for almost all $\omega \in \Omega$.
      \end{enumerate}
      If $X$ is $1$-complemented in $X^{\star\star}$, in particular if $X$ is a dual space, then $\sup \limits_{\omega \in \Omega} \|p(\omega)\| = \|P\|$.
    \end{corollary}

    \begin{proof}
      The proof of Corollary~\ref{c1} shows that $P$ commutes with the operators of multiplication by functions from $L^\infty(\Omega, \Sigma, \mu)$. Therefore (a) and (b) follow from Theorem~\ref{t2}, and the last statement follows from (a) with $\|Q\| = 1$ (for dual spaces see Corollary~\ref{c4}). Let $x \in E(X)$. Applying (b) to $x$ and to $Px$ we obtain $p(\omega)x(\omega) = Px(\omega) = P(Px)(\omega) = p(\omega)(Px)(\omega) = p(\omega)^2x(\omega)$ for almost all $\omega \in \Omega$. To obtain the last statement of (c), apply this to the constant function $\vec{y}$, which belongs to $E(X)$ when $\mu$ is finite; in the $\sigma$-finite case apply it to $\chi_{G_n}\vec{y}$ for a sequence of sets $G_n$ of finite measure with $\bigcup_n G_n = \Omega$.
    \end{proof}

    \begin{remark} \label{rperiodic}
      The aperiodicity assumption in Corollary~\ref{c1} cannot be omitted. Let $\Omega = \{1,2\}$, $\Sigma = 2^{\Omega}$, $\mu(\{1\}) = \mu(\{2\}) = 1$, let $\varphi(1) = 2$, $\varphi(2) = 1$, and let $E = L^2(\mu)$, $X = \mathds{C}$, $a \equiv 1$. Then $\mu$ is ergodic with respect to $\varphi$, and $E(X)$ can be identified with $\mathds{C}^2$, while $T(f_1, f_2) = (f_2, f_1)$. Thus $\sigma(T) = \{1, -1\}$ has two connected components although $\dim X = 1$, and the spectral projection $P = \frac{1}{2}\begin{pmatrix} 1 & 1 \\ 1 & 1\end{pmatrix}$ corresponding to $\{1\}$ does not commute with $L^\infty(\mu)$. Hence the conclusions of Corollaries~\ref{c1}, \ref{c1a}, \ref{c2} and \ref{c3} fail.

      The aperiodicity in Definition~\ref{daperiodic} is a property of the action of $\varphi$ on the measure algebra, and it cannot be replaced by the pointwise condition $\mu(\{\omega \in \Omega : \varphi^k(\omega) = \omega\}) = 0$, $k \in \mathds{N}$. Indeed, let $\Omega = \mathds{Z} \times \{0,1\}$, $\Omega_i = \mathds{Z} \times \{i\}$, $i = 0, 1$, $\Sigma = \{\emptyset, \Omega_0, \Omega_1, \Omega\}$, $\mu(\Omega_0) = \mu(\Omega_1) = 1$, and $\varphi(n,i) = (n+1, 1-i)$. Then $\varphi$ has no periodic points at all, and $\mu$ is ergodic with respect to $\varphi$. However, $\varphi(\Omega_0) = \Omega_1$ and $\varphi(\Omega_1) = \Omega_0$, so that $\varphi^2(\Omega_0) = \Omega_0$ and $\varphi$ is not aperiodic. Every measurable function is constant on $\Omega_0$ and on $\Omega_1$, and for $E = L^2(\mu)$, $X = \mathds{C}$, $a \equiv 1$ we obtain the same operator $T(f_1, f_2) = (f_2, f_1)$ on $\mathds{C}^2$ as in the previous example. Thus the conclusions of Corollaries~\ref{c1}, \ref{c1a}, \ref{c2} and \ref{c3} fail, although $\varphi$ has no periodic points.

      Note that if $\mu$ is ergodic and has no atoms, then $\varphi$ is automatically aperiodic. Indeed, otherwise by Remark~\ref{rfrolik} there are $k \in \mathds{N}$ and $G$, $\mu(G) > 0$, such that $\varphi^k(H) = H$ (up to sets of measure zero) for every measurable $H \subset G$. The set $\bigcup_{j=0}^{k-1}\varphi^j(G)$ is invariant, and by ergodicity it has a complement of measure zero; hence $\varphi^k(H) = H$ up to sets of measure zero for every $H \in \Sigma$. Then $\nu = \sum_{j=0}^{k-1}\mu \circ \varphi^j$ is a $\varphi$-invariant measure equivalent to $\mu$, and for a set $G$ with $0 < \nu(G) < \nu(\Omega)/k$ (or $0 < \nu(G) < \infty$ if $\nu(\Omega) = \infty$) the invariant set $\bigcup_{j=0}^{k-1}\varphi^j(G)$ has positive measure and a complement of positive measure, which contradicts ergodicity.
    \end{remark}

    \begin{remark} \label{r1a}
      In Corollary~\ref{c1a} the exceptional set of measure zero in (c) depends on $y$. If $X$ is not separable, we do not know whether $p$ can be chosen so that $p(\omega)$ is a projection on $X$ for almost all $\omega$.
    \end{remark}

    \begin{corollary} \label{c2}
      Assume the conditions of Corollary~\ref{c1}. Assume additionally that $\mu$ is ergodic with respect to $\varphi$ and that
      $\dim X = m < \infty$. Then the number of connected components of $\sigma(T)$ cannot exceed $m$.
    \end{corollary}

    \begin{proof} Assume that $\sigma(T) = \sigma_1 \cup \dots \cup \sigma_k$, where $\sigma_j$ are nonempty, closed and pairwise disjoint, and let $P_j$ be the corresponding spectral projections. Then $P_j \neq 0$, $P_iP_j = 0$ for $i \neq j$, $P_1 + \dots + P_k = I$, and $P_jT = TP_j$. By Corollary~\ref{c1}, $P_jx(\omega) = p_j(\omega)x(\omega)$ a.e., and, since $X$ is finite-dimensional, after removing a set of measure zero we may assume that $p_i(\omega)p_j(\omega) = 0$ for $i \neq j$ and $\sum_j p_j(\omega) = I$ for all $\omega \in \bar{\Omega}$. Note that the functions $a$ and $p_j$ are measurable, because for $e \in X$ and $G \in \Sigma$ with $\mu(G) < \infty$ and $\mu(\varphi(G)) < \infty$ we have $a(\omega)e = (T\chi_{\varphi(G)}\vec{e})(\omega)$ and $p_j(\omega)e = (P_j\chi_G\vec{e})(\omega)$ for almost all $\omega \in G$.

First, $a(\omega)$ is invertible for almost all $\omega$. Indeed, otherwise there are a set $G \in \Sigma$ of positive measure such that $\mu(\varphi(G)) < \infty$ and a measurable function $v : G \to X$ with $\|v(\omega)\| = 1$ and $a(\omega)v(\omega) = 0$, $\omega \in G$. (The set where $\det a(\omega) = 0$ is measurable, and $v$ exists by a measurable selection theorem, see e.g.~\cite[Chapter III]{CV}, because $\mu$ is complete and the multifunction $\omega \mapsto \{v \in X : a(\omega)v = 0, \ \|v\| = 1\}$ has nonempty compact values and a measurable graph on this set.) The function $x$ equal to $v(\varphi^{-1}(\eta))$ for $\eta \in \varphi(G)$ and to $0$ otherwise is a nonzero element of $E(X)$ with $Tx = 0$, which contradicts the invertibility of $T$.

Applying $P_jT = TP_j$ to functions of the form $\chi_{\varphi(G)}\vec{e}$, $e \in X$, $\mu(\varphi(G)) < \infty$, we obtain $p_j(\omega)a(\omega) = a(\omega)p_j(\varphi(\omega))$ for almost all $\omega$, and therefore $\operatorname{rank} p_j(\omega) = \operatorname{rank} p_j(\varphi(\omega))$ almost everywhere. The function $\omega \mapsto \operatorname{rank} p_j(\omega) = \operatorname{tr} p_j(\omega)$ is measurable and $\varphi$-invariant, and since $\mu$ is ergodic with respect to $\varphi$, it is equal almost everywhere to a constant $r_j$. Because $P_j \neq 0$, we have $r_j \geq 1$. Hence $k \leq r_1 + \dots + r_k = m$.
\end{proof}

 We do not know if the statement of Corollary~\ref{c2} remains valid without the assumption that the operator $T$ is invertible. A partial answer is given by our next corollary.

    \begin{corollary} \label{c3}
      The statement of Corollary~\ref{c2} remains valid without the assumption that the operator $T$ is invertible if we assume that the composition operator $V$, $Vx(\omega) = x (\varphi(\omega))$, is bounded and invertible on $E(X)$.
    \end{corollary}

    \begin{proof}
  Since $T = M_aV$, where $M_a$ is the operator of multiplication by $a$, the operator $M_a = TV^{-1}$ is bounded on $E(X)$, and the argument used to prove~(\ref{eq3}) shows that $\|a(\omega)\| \leq \|M_a\|$ for almost all $\omega$. Thus $a$ belongs to $L^\infty(\mu; M_m(\mathds{C}))$, which is a finite von Neumann algebra, and therefore $a$ can be approximated in the norm of this algebra by invertible elements $a_n$ (see \cite[V.3.1.5, V.3.1.16, pp. 446, 449]{Bl}). Since $\|M_{a_n - a}\| \leq \|a_n - a\|_\infty$, the operators $T_n = M_{a_n}V$, $T_nx(\omega) = a_n(\omega)x(\varphi(\omega))$, are invertible, satisfy the assumptions of Corollary~\ref{c2}, and converge to $T$ in the operator norm. Assume that $\sigma(T)$ has more than $m$ connected components. Then $\sigma(T) = \sigma_1 \cup \dots \cup \sigma_{m+1}$, where the sets $\sigma_j$ are nonempty, closed and pairwise disjoint. Let $U_1, \dots, U_{m+1}$ be pairwise disjoint open neighborhoods of $\sigma_1, \dots, \sigma_{m+1}$. By the upper semicontinuity of the spectrum, $\sigma(T_n) \subset U_1 \cup \dots \cup U_{m+1}$ for all sufficiently large $n$. Moreover, by the continuity of Riesz projections \cite{Ne}, the spectral projections of $T_n$ corresponding to $\sigma(T_n) \cap U_j$ converge in norm to the (nonzero) spectral projections of $T$ corresponding to $\sigma_j$. Hence they are nonzero for large $n$, that is, $\sigma(T_n) \cap U_j \neq \emptyset$, $j = 1, \dots, m+1$. Therefore $\sigma(T_n)$ has at least $m+1$ connected components, which contradicts Corollary~\ref{c2} applied to $T_n$.
\end{proof}

     \begin{remark} \label{r0}
        A result similar to Corollary~\ref{c2} is~\cite[Theorem 9.5]{An}.  It states the following. Let $C^\star(\mathfrak{A},U)$ be the $C^\star$-algebra generated by a $C^\star$-algebra $\mathfrak{A}$ and a unitary operator $U$ such that $U\mathfrak{A}U^{-1} = \mathfrak{A}$. Assume that $\mathfrak{A}$ is isomorphic to the algebra of endomorphisms of an $N$-dimensional complex vector bundle over a compact space $M$, that $\mathds{Z}$ acts topologically freely on $\mathfrak{A}$, and that $M$ is irreducible with respect to the homeomorphism $\alpha$ of $M$ induced by $U$. Then for every $a \in \mathfrak{A}$ the spectrum of $aU$ has at most $N$ connected components. For scalar coefficients ($N = 1$) see also~\cite[Theorem 7.2, Corollary 7.4, Lemma 8.3]{AL}.

    Neither of these two results contains the other. Corollary~\ref{c2} applies to arbitrary K\"othe--Bochner spaces $E(X)$, measurable coefficients, and transformations $\varphi$ that need not preserve $\mu$, whereas \cite[Theorem 9.5]{An} is a statement about $C^\star$-algebras of operators on a Hilbert space. On the other hand, \cite[Theorem 9.5]{An} applies to nontrivial vector bundles and to non-invertible coefficients, and its irreducibility assumption is topological rather than measure-theoretic. In this sense the two results complement each other.
    \end{remark}

    The results of this section were obtained by combining the general
representation Theorems~\ref{t1} and \ref{t2} with the simple fact that
the spectral projections commute with $L^\infty$; the specific structure
of a spectral projection, namely, that it is a contour integral of the
resolvent of a weighted composition operator, has not been used. In the next section we
show that this structure can be exploited to remove the assumptions on
the Banach space $X$ in the case when $E=L^p(\mu)$, $1 \leq p < \infty$.

    \section{Spectral projections. The spaces $L^p(X)$} \label{s4}

By Corollaries~\ref{c1} and \ref{c1a}, if $X$ is either separable or complemented in $X^{\star\star}$ then every spectral projection of a weighted composition operator on $E(X)$ is an operator of multiplication by an $L(X)$-valued function. In this section we prove the same conclusion for an \emph{arbitrary} Banach space $X$, at the price of the following restrictions: $E = L^p(\mu)$, $1 \leq p < \infty$; the transformation $\varphi$ is nonsingular and aperiodic; and the weight, divided by the factor that makes the composition operator an isometry of $L^p$, is measurable in the norm of $L(X)$.
Theorem~\ref{t4} below and its proof were inspired by ideas of C.~Chicone and Y.~Latushkin (see~\cite[Section 4.1]{CL}) but the theorem does not follow directly from their results. There the coefficients are continuous functions on $\mathds{R}$, and continuity is used to pass from an operator on $L^p(\mathds{R};X)$ to its restrictions to individual orbits. In our situation there is no topology, and this role is played by the norm-measurability of the weight.

The general scheme of the proof goes back to A.~Antonevich and A.~Lebedev. Let $C^\star(\mathfrak{A},U)$ be the $C^\star$-algebra generated by a $C^\star$-algebra $\mathfrak{A}$ and a unitary operator $U$ with $U\mathfrak{A}U^{-1} = \mathfrak{A}$, and assume that the action of $\mathds{Z}$ on $\mathfrak{A}$ is topologically free. If $a \in \mathfrak{A}$ and $\sigma(aU) \cap \mathds{T} = \emptyset$, then the spectral projection of $aU$ corresponding to the part of the spectrum inside the unit disk belongs to $\mathfrak{A}$; see \cite[Lemma 19.1]{AL} and, for algebras of endomorphisms of finite-dimensional vector bundles, the necessity part of the proof of \cite[Theorem 9.2, pp.~81--82]{An}. In these proofs, as in ours, the spectral projection is written as a contour integral of the resolvent, it is shown to commute with the coefficient algebra, and then the conditional expectation (the von Neumann mapping $N$) is used to show that it belongs to the coefficient algebra. For a commutative coefficient algebra $\mathfrak{A} = C(M)$ acting isometrically on a Banach space, the same conclusion is \cite[Theorem 7.2]{AL}.

For scalar coefficients, and more generally whenever the commutant of $L^\infty(\mu)$ consists of multiplication operators (e.g., if $X$ is separable or complemented in $X^{\star\star}$, see Theorems~\ref{t1} and~\ref{t2}), a representation $(Px)(\omega) = P(\omega)x(\omega)$ with $P(\omega) \in L(X)$ follows from the fact that $P$ commutes with $L^\infty(\mu)$ (cf.\ Corollaries~\ref{c1} and~\ref{c1a} and \cite[Theorem 7.2]{AL}). For an arbitrary Banach space $X$ this argument is not available: an operator on $L^p(X)$ commuting with $L^\infty(\mu)$ is known only to be given by a function with values in $L(X,X^{\star\star})$ (Remark~\ref{r1}). We therefore have to show, in addition, that $P$ belongs to the closed algebra generated by the operators of multiplication by norm-measurable $L(X)$-valued functions and the isometry $V$ defined in (\ref{eqV}). In the $C^\star$-setting of \cite{AL,An} the corresponding step is automatic, because $C^\star$-subalgebras are inverse closed. In our setting this algebra is not a $C^\star$-algebra, its coefficients are noncommuting operators on an arbitrary Banach space, and topological freeness is replaced by aperiodicity of $\varphi$. The inverse closedness needed for the resolvent of the weighted composition operator is the main technical result of this section (Proposition~\ref{p310}).

\subsection{Notation and statement}

Let $(\Omega, \Sigma, \mu)$ be a complete measure space with a $\sigma$-finite measure $\mu$, and let $Y$ be a Banach space. For $1 \leq p < \infty$ we write $L^p(Y) = L^p(\Omega,\Sigma,\mu;Y)$.

Let $L^\infty_u(\Omega; L(Y))$ be the K\"othe--Bochner space $L^\infty(\Omega, \Sigma, \mu; L(Y))$, where $L(Y)$ is regarded as a Banach space with the operator norm. Thus $b \in L^\infty_u(\Omega; L(Y))$ if $b : \Omega \to L(Y)$ is strongly measurable, i.e., $b$ is an almost everywhere limit in the operator norm of finitely valued measurable functions, and $\|b\|_\infty = \operatorname{ess\,sup}_\omega \|b(\omega)\| < \infty$. Functions equal almost everywhere are identified. The index $u$ indicates that measurability is understood with respect to the uniform operator topology. With the pointwise product $L^\infty_u(\Omega; L(Y))$ is a Banach algebra. For such $b$ we put $M_b x(\omega) = b(\omega)x(\omega)$. We say that a function $b : \Omega \to L(Y)$ is \emph{norm-measurable} if it belongs to $L^\infty_u(\Omega;L(Y))$.

By~\cite[Lemma 3.2]{DKO} applied with $E = L^\infty(\mu)$ and $X = L(Y)$, countably valued functions $\sum_j \chi_{G_j} B_j$, where $B_j \in L(Y)$ and $G_j \in \Sigma$ are pairwise disjoint, are dense in $L^\infty_u(\Omega; L(Y))$. In other words, every $b \in L^\infty_u(\Omega; L(Y))$ is, outside a set of measure zero, a uniform limit of such functions. Conversely, an essentially bounded function with this property belongs to $L^\infty_u(\Omega; L(Y))$, because countably valued essentially bounded functions belong to this space and it is complete.

\begin{remark} \label{r30}
Norm-measurability is strictly stronger than measurability in the strong operator topology, i.e., than measurability of all functions $\omega \mapsto b(\omega)y$, $y \in Y$. Indeed, the function $b : [0,1] \to L(L^2[0,1])$ of Remark~\ref{rnotnorm} is continuous in the strong operator topology, but $b \notin L^\infty_u([0,1]; L(L^2[0,1]))$. Remark~\ref{rnotnorm} also shows that the representing functions provided by Theorem~\ref{t1} need not be norm-measurable.
\end{remark}

Let $\varphi : \Omega \to \Omega$ be an invertible measurable nonsingular map, and recall the notions of $N$-separated sets and of aperiodicity from Definition~\ref{daperiodic}. If $\varphi$ is aperiodic, then for every $N$ every set of positive measure contains an $N$-separated subset of finite positive measure. In the terminology of \cite[p.~74]{An}, the action of $\mathds{Z}$ on $L^\infty(\mu)$ generated by $\varphi$ is then metrically free.

Let $\rho = d(\mu \circ \varphi)/d\mu$ be the Radon--Nikodym derivative of the measure $G \to \mu(\varphi(G))$ with respect to $\mu$; then $0 < \rho < \infty$ almost everywhere. For $1 \leq p < \infty$ let
\begin{equation} \label{eqV}
(Vx)(\omega) = \rho(\omega)^{1/p}\, x(\varphi(\omega)), \quad x \in L^p(Y), \ \omega \in \Omega.
\end{equation}
Since $\int \|x(\varphi(\omega))\|^p \rho(\omega)\,d\mu(\omega) = \int \|x\|^p\,d\mu$, $V$ is an invertible isometry of $L^p(Y)$, and $(V^{-1}x)(\omega) = \rho(\varphi^{-1}(\omega))^{-1/p}x(\varphi^{-1}(\omega))$. If $\varphi$ preserves $\mu$, then $\rho = 1$ and $V$ is the composition operator $x \mapsto x \circ \varphi$. The operator $V$ depends on $p$; when it is used on $L^q$ we always mean the operator (\ref{eqV}) with $q$ in place of $p$. Since $\rho$ is a scalar function, for every (scalar or operator-valued) multiplier $c$
\begin{equation} \label{eqrules}
V^k M_c = M_{c \circ \varphi^k} V^k, \quad k \in \mathds{Z}.
\end{equation}

\begin{theorem} \label{t4}
Let $(\Omega, \Sigma, \mu)$ be a complete measure space with a $\sigma$-finite measure $\mu$, let $1 \leq p < \infty$, and let $\varphi : \Omega \to \Omega$ be an invertible aperiodic nonsingular transformation. Let $X$ be an arbitrary Banach space, let $A \in L^\infty_u(\Omega; L(X))$, and let
\[ (Tx)(\omega) = A(\omega)\rho(\omega)^{1/p}x(\varphi(\omega)), \quad x \in L^p(X), \ \omega \in \Omega, \]
that is, $T = M_AV$.
Assume that $T$ is invertible and that $\sigma(T) = \sigma_1 \cup \sigma_2$, where $\sigma_1 \subset \mathds{U}$ and $\sigma_2 \subset \mathds{C} \setminus \mathds{D}$. Let $P$ be the spectral projection of $T$ corresponding to $\sigma_1$. Then there is a norm-measurable function $P(\cdot) : \Omega \to L(X)$ such that $P(\omega)^2 = P(\omega)$ and $\|P(\omega)\| \leq \|P\|$ for every $\omega \in \Omega$, and for every $x \in L^p(X)$
\[ (Px)(\omega) = P(\omega)x(\omega) \quad \text{for almost all } \omega \in \Omega. \]
\end{theorem}

Equivalently, Theorem~\ref{t4} applies to every weighted composition operator $(Tx)(\omega) = a(\omega)x(\varphi(\omega))$ whose weight has the form $a = \rho^{1/p}A$ with $A \in L^\infty_u(\Omega;L(X))$. Note that the composition operator $x \mapsto x \circ \varphi$ itself need not be bounded on $L^p(X)$. A weight $a = hA$ with a scalar function $h$ such that $|h| = \rho^{1/p}$ reduces to this case, because $(h/|h|)A \in L^\infty_u(\Omega;L(X))$.

We write $\mathcal{A}_u = \{M_b : b \in L^\infty_u(\Omega;L(X))\}$, and we denote by $\mathcal{B}$ the closure in $L(L^p(X))$ of the set $\mathcal{B}_0$ of all finite sums $\sum_{|k| \leq N} M_{b_k}V^k$, $b_k \in L^\infty_u(\Omega; L(X))$. By (\ref{eqrules}), $\mathcal{B}_0$ is an algebra, so $\mathcal{B}$ is a closed subalgebra of $L(L^p(X))$. Note that $T = M_A V \in \mathcal{B}_0$.

\subsection{Outline of the proof} \label{ssoutline}
The proof of Theorem~\ref{t4} consists of four steps. We prove the auxiliary results in the subsequent subsections.

\begin{proof}[Proof of Theorem~\ref{t4}]
\emph{Step 1: $P \in \mathcal{B}$.} This is Corollary~\ref{c311}. It is the main technical part of the proof. We show that $(\lambda - T)^{-1} \in \mathcal{B}$ for $|\lambda| = 1$ (Proposition~\ref{p310}), and then $P \in \mathcal{B}$ because $P$ is the integral of the resolvent over the unit circle. To prove Proposition~\ref{p310} we restrict $I - \lambda^{-1}T$ to the orbits of $\varphi$. This gives a family of operators $\pi_\omega$ on $\ell^p(\mathds{Z};X)$. These operators are uniformly invertible for almost all $\omega$ (Lemmas~\ref{l36} and \ref{l37}). The matrix entries of their inverses decay exponentially off the diagonal (Lemma~\ref{l38}) and depend continuously on finitely many values of the weight along the orbit (Lemma~\ref{l39}). This allows us to write $(I - \lambda^{-1}T)^{-1}$ as an absolutely convergent series $\sum_k M_{c_k}V^k$ with norm-measurable coefficients $c_k$.

\emph{Step 2: $P$ commutes with $L^\infty(\mu)$.} This is Lemma~\ref{l35}. The range of $P$ consists of the vectors $x$ with $\|T^n x\| \to 0$, and its kernel consists of the vectors $x$ with $\|T^{-n}x\| \to 0$. Both subspaces are invariant under multiplication by bounded functions.

\emph{Step 3: $P$ is a multiplication operator.} An element of $\mathcal{B}$ that commutes with $L^\infty(\mu)$ belongs to $\mathcal{A}_u$ (Lemma~\ref{l34}). By Steps 1 and 2, $P = M_b$ for some $b \in L^\infty_u(\Omega;L(X))$. Lemma~\ref{l34} relies on the ``conditional expectation'' $N_0$ of Lemma~\ref{l33}, and this is where aperiodicity of $\varphi$ is used.

\emph{Step 4: choice of a representative.} By Lemma~\ref{l312}, $b$ can be chosen so that $b(\omega)$ is a projection and $\|b(\omega)\| \leq \|P\|$ for every $\omega \in \Omega$. It remains to put $P(\omega) = b(\omega)$.
\end{proof}

The structure of this proof ($P$ belongs to the algebra generated by the coefficients and the composition operator, $P$ commutes with the coefficients, hence $P$ belongs to the coefficient algebra) is the same as in the proofs of \cite[Lemma 19.1]{AL} and \cite[Theorem 9.2]{An}; compare Remarks~\ref{r31} and \ref{r32}.

\subsection{Elements of $\mathcal{B}$ commuting with $L^\infty$}

\begin{lemma} \label{l31}
For $b \in L^\infty_u(\Omega;L(Y))$ we have $\|M_b\| = \|b\|_\infty$ on $L^p(Y)$. Consequently $\{M_b : b \in L^\infty_u(\Omega; L(Y))\}$ is a closed subalgebra of $L(L^p(Y))$.
\end{lemma}

\begin{proof}
Clearly $M_b$ maps $L^p(Y)$ into itself and $\|M_b\| \leq \|b\|_\infty$. Conversely, let $\delta > 0$ and let $b' = \sum_j \chi_{G_j} B_j$ be countably valued with $\|b - b'\|_\infty < \delta$. There is a $j$ with $\mu(G_j) > 0$ and $\|B_j\| > \|b'\|_\infty - \delta$. Choose $G \subset G_j$ with $0 < \mu(G) < \infty$, choose $y \in Y$, $y \neq 0$, with $\|B_j y\| > (\|b'\|_\infty - \delta)\|y\|$, and put $x = \chi_{G}\vec{y} \in L^p(Y)$. Then $\|M_{b'}x\| > (\|b'\|_\infty - \delta)\|x\|$, and therefore $\|M_b\| \geq \|M_{b'}\| - \delta \geq \|b\|_\infty - 3\delta$.
\end{proof}

\begin{lemma} \label{l32}
Let $\mathcal{F} \subset \Sigma$ be such that every set of positive measure contains a member of $\mathcal{F}$ of positive measure. Let $h, k \geq 0$ be functions in $L^p(\mu)$ and $c > 0$, and assume that $\|\chi_F h\|_p \leq c\|\chi_F k\|_p$ for every $F \in \mathcal{F}$. Then $h \leq 2ck$ almost everywhere.
\end{lemma}

\begin{proof}
Let $G = \{h > 2ck\}$ and assume $\mu(G) > 0$. Choose $F \in \mathcal{F}$, $F \subset G$, $\mu(F) > 0$. Then $\|\chi_F h\|_p \geq 2c\|\chi_F k\|_p \geq 2\|\chi_F h\|_p$, so $\chi_F h = 0$ a.e., which contradicts $h > 0$ on $F$.
\end{proof}

\begin{lemma} \label{l33}
Let $b = \sum_{|k| \leq N} M_{b_k}V^k \in \mathcal{B}_0$. Then $\chi_F b \chi_F = \chi_F M_{b_0}$ for every $N$-separated set $F$, and $\|M_{b_0}\| \leq 2\|b\|$. Consequently the map $N_0 : b \mapsto M_{b_0}$ is well defined on $\mathcal{B}_0$, and it extends to a bounded linear map $N_0 : \mathcal{B} \to \mathcal{A}_u$ such that $\|N_0\| \leq 2$ and $N_0(D) = D$ for $D \in \mathcal{A}_u$.
\end{lemma}

\begin{proof}
By (\ref{eqrules}), $M_{b_k}V^k\chi_F = M_{b_k}(\chi_F \circ \varphi^k)V^k$, and $\chi_F \cdot (\chi_F \circ \varphi^k) = \chi_{F \cap \varphi^{-k}(F)} = 0$ a.e. for $0 < |k| \leq N$. Since multipliers commute with $\chi_F$, we obtain $\chi_F b \chi_F = \chi_F M_{b_0}$. Fix $x \in L^p(X)$ and put $h(\omega) = \|b_0(\omega)x(\omega)\|$, $k(\omega) = \|b\|\,\|x(\omega)\|$. For every $N$-separated $F$ we have $\|\chi_F h\|_p = \|\chi_F b \chi_F x\| \leq \|b\|\,\|\chi_F x\| = \|\chi_F k\|_p$. Because $\varphi$ is aperiodic, Lemma~\ref{l32} applies and gives $\|M_{b_0}x\| \leq 2\|b\|\,\|x\|$. The remaining statements follow from this inequality and Lemma~\ref{l31}.
\end{proof}

\begin{lemma} \label{l34}
If $D \in \mathcal{B}$ and $D\chi_F = \chi_F D$ for every $F \in \Sigma$, then $D = N_0(D) \in \mathcal{A}_u$.
\end{lemma}

\begin{proof}
Put $D' = D - N_0(D)$. Then $D' \in \mathcal{B}$, $D'$ commutes with every $\chi_F$, and $N_0(D') = 0$. Let $\varepsilon > 0$ and choose $b \in \mathcal{B}_0$ with $\|D' - b\| < \varepsilon$; let $N$ be such that $b = \sum_{|k| \leq N} M_{b_k}V^k$. Then $\|N_0(b)\| = \|N_0(b - D')\| \leq 2\varepsilon$. For an $N$-separated $F$ and $x \in L^p(X)$ we have, by Lemma~\ref{l33},
\[ \chi_F D'x = \chi_F D' \chi_F x = \chi_F(D' - b)\chi_F x + \chi_F N_0(b) x, \]
hence $\|\chi_F D'x\| \leq 3\varepsilon\|\chi_F x\|$. By Lemma~\ref{l32}, $\|(D'x)(\omega)\| \leq 6\varepsilon\|x(\omega)\|$ a.e., so $\|D'\| \leq 6\varepsilon$. Thus $D' = 0$.
\end{proof}

\begin{remark} \label{r31}
The map $N_0$ is an analogue of the von Neumann mapping (conditional expectation) $N(\sum a_g T_g) = a_e$ (see \cite[p.~61]{An} and \cite[Subsection 12.2, p.~166]{AL}), and the inequality $\|M_{b_0}\| \leq 2\|b\|$ of Lemma~\ref{l33} is an analogue of the inequality $\|\sum a_g T_g\| \geq \|a_e\|$. This inequality is (7.13) in \cite[Lemma 7.1, p.~59]{An}; in \cite[Subsection 12.2, p.~166]{AL} it is called property $(*)$. In the $C^\star$-setting it is derived from topological freeness of the action; see \cite[Lemma 7.2, p.~60]{An} and \cite[Theorem 12.14, p.~173]{AL}. Our proof uses $N$-separated sets in the same way as the proof of \cite[Lemma 7.2]{An} uses neighborhoods $W$ with $\alpha_g(W) \cap \alpha_h(W) = \emptyset$. Aperiodicity of $\varphi$ plays the role of topological freeness. For operators of multiplication by functions from $L^\infty$ the relation between the two conditions is discussed in \cite[Theorem 8.12]{An} (``metrically free'' actions; see also \cite[Notes to Section 12, p. 226]{AL}). Lemma~\ref{l34} is the counterpart of the fact that in the $C^\star$-setting an element is determined by its coefficients $N(bT_g^{-1})$ \cite[Theorem 8.11]{An}, \cite[Theorem 12.4]{AL}.
\end{remark}

\subsection{The commutation lemma}

\begin{lemma} \label{l35}
Under the assumptions of Theorem~\ref{t4}, $P$ commutes with the operators of multiplication by functions from $L^\infty(\mu)$.
\end{lemma}

\begin{proof}
Let $E_1 = PL^p(X)$ and $E_2 = (I-P)L^p(X)$. Both are invariant under $T$, the spectral radius of $T|_{E_1}$ is less than $1$, and, since $T$ is invertible, so is the spectral radius of $T^{-1}|_{E_2}$. As in the proof of Corollary~\ref{c1},
\[ E_1 = \{x : \|T^n x\| \to 0\}, \qquad E_2 = \{x : \|T^{-n}x\| \to 0\} \quad (n \to \infty). \]
(If $\|T^n x\| \to 0$, then $(I-P)x = (T^{-1}|_{E_2})^n(I-P)T^n x \to 0$; the second equality is proved in the same way.) For $g \in L^\infty(\mu)$ we have $T^n(gx) = (g \circ \varphi^n)T^n x$ for all $n \in \mathds{Z}$, and therefore $\|T^n(gx)\| \leq \|g\|_\infty \|T^n x\|$. Hence $gE_1 \subset E_1$ and $gE_2 \subset E_2$, which means $gP = Pg$.
\end{proof}

The characterization of $E_1$ used above and the resulting invariance of $E_1$ under the coefficient algebra are the same as in Step~I of the proof of \cite[Theorem 7.2]{AL} and in the proof of \cite[Theorem 9.2, p.~81]{An}. There the invariance of $E_2$ is obtained by duality (\cite[Lemma 7.3]{AL}) or from invertibility of $a_1$ followed by an approximation argument; here it follows directly from the invertibility of $T$.

\subsection{The resolvent of $T$ belongs to $\mathcal{B}$}\label{ss34}

Elements of $L^\infty_u$ are equivalence classes. For each coefficient $b$ that occurs below we fix a representative with $\|b(\omega)\| \leq \|b\|_\infty$ for every $\omega \in \Omega$, and a countable set $\mathcal{D}_b \subset L(Y)$ such that $b(\omega) \in \overline{\mathcal{D}_b}$ for all $\omega$ outside a set of measure zero. (If $b^{(j)}$ are countably valued and $\|b - b^{(j)}\|_\infty \leq 1/j$, we can take for $\mathcal{D}_b$ the set of all values of all $b^{(j)}$.) Since $\varphi$ and $\varphi^{-1}$ map sets of measure zero onto sets of measure zero, every set $N$ of measure zero is contained in the $\varphi$-invariant set $\bigcup_{n \in \mathds{Z}} \varphi^n(N)$ of measure zero. All exceptional sets below are chosen $\varphi$-invariant in this way.

For $b = \sum_{|k| \leq L} M_{b_k}V^k$ and $\omega \in \Omega$ we define the \emph{orbit operator} $\pi_\omega(b)$ on $\ell^p(\mathds{Z};Y)$ by
\[ (\pi_\omega(b)\xi)_m = \sum_{|k| \leq L} b_k(\varphi^m(\omega))\,\xi_{m+k}, \quad m \in \mathds{Z}. \]
In other words, if $(S\xi)_n = \xi_{n+1}$ is the backward shift on $\ell^p(\mathds{Z};Y)$, then
\[ \pi_\omega(b) = \sum_{|k| \leq L} \operatorname{diag}\big(b_k(\varphi^n(\omega))\big)_{n \in \mathds{Z}}\, S^k, \]
that is, $\pi_\omega$ replaces the composition operator $V$ by $S$ and the multiplier $M_c$ by the diagonal operator $\operatorname{diag}(c(\varphi^n(\omega)))_n$, the values of $c$ along the orbit of $\omega$. Clearly $\|\pi_\omega(b)\| \leq \sum_k \|b_k\|_\infty$. The operators $\pi_\omega(b)$ are the measurable analogue of the trajectorial (orbital) representations $\pi_\tau$ of $C^\star$-algebras generated by dynamical systems; see \cite[Section 10, formulas (10.1)--(10.2) and Theorem 10.1]{An} and \cite[Section 21, formulas (21.1)--(21.2) and Theorem 21.2]{AL}. In the $C^\star$-setting $\pi_\tau$ is a $\star$-representation, so the invertibility of $\pi_\tau(b)$ for invertible $b$ and the bound $\|\pi_\tau(b)^{-1}\| \leq \|b^{-1}\|$ are automatic. In our situation these facts are given by Lemmas~\ref{l36} and \ref{l37}. The role of a uniform bound for the inverses of the orbit operators is discussed in \cite[Remark 10.1, p.~93]{An}. For an operator $C$ on $\ell^p(\mathds{Z};Y)$ we denote by $C[n,m] \in L(Y)$ its matrix entries, $C[n,m]y = (C(e_m \otimes y))_n$, where $e_m \otimes y$ is the sequence with $y$ in place $m$ and $0$ elsewhere.

\begin{lemma} \label{l36}
Let $1 \leq p < \infty$, let $Y$ be a Banach space, let $b = \sum_{|k| \leq L} M_{b_k}V^k$ with $b_k \in L^\infty_u(\Omega;L(Y))$, and assume that $\|bx\| \geq c\|x\|$ for all $x \in L^p(Y)$. Then there is a set $N$ of measure zero such that $\|\pi_\omega(b)\xi\| \geq c\|\xi\|$ for all $\omega \notin N$ and all $\xi \in \ell^p(\mathds{Z};Y)$.
\end{lemma}

\begin{proof}
Let $\mathcal{D} = \bigcup_k \mathcal{D}_{b_k}$, and let $N'$ be a set of measure zero such that $b_k(\varphi^m(\omega)) \in \overline{\mathcal{D}}$ for all $k$, $m$ and $\omega \notin N'$. Fix $M \in \mathds{N}$ and $\varepsilon > 0$. For every family $\beta = (\beta_{k,m})_{|k| \leq L,\, |m| \leq M+L}$ of elements of $\mathcal{D}$ let
\[ G_\beta = \{\omega \notin N' : \|b_k(\varphi^m(\omega)) - \beta_{k,m}\| \leq \varepsilon \text{ for } |k| \leq L, \ |m| \leq M+L\}. \]
Since the functions $b_k$ are norm-measurable, these sets are measurable; there are countably many of them, and they cover $\Omega \setminus N'$. Let $N_{M,\varepsilon}$ be the union of those $G_\beta$ that have measure zero, and let $N = N' \cup \bigcup_{M, j \in \mathds{N}} N_{M,1/j}$.

Assume $\mu(G_\beta) > 0$ and let $\xi \in \ell^p(\mathds{Z};Y)$ be supported in $[-M, M]$. Define $\pi_\beta\xi$ by $(\pi_\beta \xi)_m = \sum_{|k| \leq L}\beta_{k,m}\xi_{m+k}$ for $|m| \leq M+L$ and $(\pi_\beta\xi)_m = 0$ otherwise. We claim that
\begin{equation} \label{eqpattern}
\|\pi_\beta \xi\| \geq (c - (2L+1)\varepsilon)\|\xi\|.
\end{equation}
Indeed, choose a $(2M+4L)$-separated set $F \subset G_\beta$ with $0 < \mu(F) < \infty$ and put $f = \sum_{|n| \leq M} V^{-n}(\chi_F\vec{\xi}_n)$. The function $V^{-n}(\chi_F\vec{\xi}_n)$ vanishes outside $\varphi^n(F)$, these sets are pairwise disjoint up to sets of measure zero for $|n| \leq M + L$, and $V$ is an isometry; hence $\|f\|^p = \mu(F)\|\xi\|^p$. Put $\xi_j = 0$ for $|j| > M$. By (\ref{eqrules}), $V^mM_cV^{-m} = M_{c \circ \varphi^m}$, and therefore
\[ bf = \sum_{|k| \leq L}\sum_{|n| \leq M} M_{b_k}V^{k-n}(\chi_F\vec{\xi}_n) = \sum_{|m| \leq M+L} V^{-m}\big(\chi_F\, g_m\big), \]
where
\[ g_m(\omega') = \sum_{|k| \leq L} b_k(\varphi^m(\omega'))\,\xi_{m+k} = (\pi_{\omega'}(b)\xi)_m. \]
Moreover $(\pi_{\omega'}(b)\xi)_m = 0$ for $|m| > M+L$. Since the functions $V^{-m}(\chi_Fg_m)$ have disjoint supports and $V$ is an isometry,
\[ \|bf\|^p = \sum_{|m| \leq M+L}\int_F \|g_m(\omega')\|^p \, d\mu(\omega') = \int_F \|\pi_{\omega'}(b)\xi\|^p \, d\mu(\omega'). \]
For $\omega' \in G_\beta$ each coordinate of $(\pi_{\omega'}(b) - \pi_\beta)\xi$ has norm at most $\varepsilon\sum_{|k| \leq L}\|\xi_{m+k}\|$, hence $\|(\pi_{\omega'}(b) - \pi_\beta)\xi\| \leq (2L+1)\varepsilon\|\xi\|$. Consequently
\[ c\,\mu(F)^{1/p}\|\xi\| = c\|f\| \leq \|bf\| \leq \mu(F)^{1/p}\big(\|\pi_\beta\xi\| + (2L+1)\varepsilon\|\xi\|\big), \]
which proves (\ref{eqpattern}). Note that (\ref{eqpattern}) concerns a fixed finite family of operators and holds for all $\xi$ simultaneously.

Now let $\omega \notin N$, $M \in \mathds{N}$ and $\varepsilon = 1/j$. Then $\omega \in G_\beta$ for some $\beta$ with $\mu(G_\beta) > 0$, and by (\ref{eqpattern}) $\|\pi_\omega(b)\xi\| \geq (c - 2(2L+1)\varepsilon)\|\xi\|$ for every $\xi$ supported in $[-M,M]$. Letting $j \to \infty$ and $M \to \infty$, and using that finitely supported sequences are dense in $\ell^p(\mathds{Z};Y)$, we obtain the lemma.
\end{proof}

\begin{remark} \label{rinfty}
Let $L^\infty(Y) = L^\infty(\Omega, \Sigma, \mu; Y)$ be the space of essentially bounded strongly measurable functions, and let $V$ denote on $L^\infty(Y)$ the composition operator $x \mapsto x \circ \varphi$, i.e., (\ref{eqV}) with $1/p = 0$; it is an invertible isometry of $L^\infty(Y)$ and (\ref{eqrules}) holds. The proof of Lemma~\ref{l36} remains valid for $p = \infty$ if the integrals are replaced by essential suprema, but only for finitely supported $\xi$. Thus, if $b = \sum_{|k| \leq L} M_{b_k}V^k$ with $b_k \in L^\infty_u(\Omega;L(Y))$ and $\|bx\|_\infty \geq c\|x\|_\infty$ for all $x \in L^\infty(Y)$, then there is a set $N$ of measure zero such that $\|\pi_\omega(b)\xi\|_\infty \geq c\|\xi\|_\infty$ for all $\omega \notin N$ and all $\xi \in c_0(\mathds{Z};Y)$. Since finitely supported sequences are not dense in $\ell^\infty(\mathds{Z};Y)$, this argument does not give the estimate on $\ell^\infty(\mathds{Z};Y)$.
\end{remark}

\begin{lemma} \label{l37}
Let $d \in L^\infty_u(\Omega;L(X))$ and assume that $D = I - M_dV$ is invertible on $L^p(X)$, $1 \leq p < \infty$. Let $K = \|D^{-1}\|$. Then there is a $\varphi$-invariant set $N$ of measure zero such that for every $\omega \notin N$ the operator $D_\omega = \pi_\omega(D) = I - \operatorname{diag}(d(\varphi^n(\omega)))_n S$ is invertible on $\ell^p(\mathds{Z};X)$ and $\|D_\omega^{-1}\| \leq K$.
\end{lemma}

\begin{proof}
By Lemma~\ref{l36} with $Y = X$, $\|D_\omega\xi\| \geq K^{-1}\|\xi\|$ for $\omega$ outside a set of measure zero.

Let $1/p + 1/q = 1$ (with $q = \infty$ if $p = 1$). Since $\mu$ is $\sigma$-finite, the space $L^q(X^\star)$ of strongly measurable functions is isometrically embedded in $L^p(X)^\star$ by $\langle x, g \rangle = \int \langle x(\omega), g(\omega)\rangle \, d\mu$. Changing the variable $\tau = \varphi(\omega)$ and using $\rho(\varphi^{-1}(\tau))^{-1} = d(\mu \circ \varphi^{-1})/d\mu\,(\tau)$, we obtain
\[ \langle M_dVx, g\rangle = \int \langle x(\tau), \rho(\varphi^{-1}(\tau))^{-1/q}\, d(\varphi^{-1}(\tau))^\star g(\varphi^{-1}(\tau))\rangle \, d\mu(\tau). \]
Hence $D^\star$ maps $L^q(X^\star)$ into itself and coincides there with $D' = I - M_{d'}V^{-1}$, where $V$ is the operator (\ref{eqV}) on $L^q(X^\star)$ (with $q$ in place of $p$; for $q = \infty$ see Remark~\ref{rinfty}) and $d'(\omega) = d(\varphi^{-1}(\omega))^\star$. The function $d'$ belongs to $L^\infty_u(\Omega;L(X^\star))$, and $\|D'g\| \geq K^{-1}\|g\|$ for $g \in L^q(X^\star)$.

Let first $1 < p < \infty$. By Lemma~\ref{l36} with $Y = X^\star$ and the exponent $q$, $\|\pi_\omega(D')\eta\| \geq K^{-1}\|\eta\|$ for $\eta \in \ell^q(\mathds{Z};X^\star)$ and $\omega$ outside a set of measure zero. On the other hand $\ell^p(\mathds{Z};X)^\star = \ell^q(\mathds{Z};X^\star)$, and
\begin{align*} \langle D_\omega\xi, \eta\rangle &= \sum_n \langle \xi_n - d(\varphi^n(\omega))\xi_{n+1}, \eta_n\rangle \\ &= \sum_n \langle \xi_n, \eta_n - d(\varphi^{n-1}(\omega))^\star\eta_{n-1}\rangle = \langle \xi, \pi_\omega(D')\eta\rangle, \end{align*}
that is, $D_\omega^\star = \pi_\omega(D')$. Thus $D_\omega$ is bounded below and $D_\omega^\star$ is injective, so $D_\omega$ is invertible and $\|D_\omega^{-1}\| \leq K$.

Let now $p = 1$, so that $q = \infty$. By Remark~\ref{rinfty} with $Y = X^\star$, there is a set of measure zero outside which the operator $B_\omega = \pi_\omega(D')$, $(B_\omega\eta)_n = \eta_n - d(\varphi^{n-1}(\omega))^\star\eta_{n-1}$, satisfies $\|B_\omega\eta\| \geq K^{-1}\|\eta\|$ for $\eta \in c_0(\mathds{Z};X^\star)$. Fix $\omega$ outside the union of the exceptional sets (which we again choose $\varphi$-invariant). Then $B_\omega$ is an injective operator on $c_0(\mathds{Z};X^\star)$ with closed range, and by the closed range theorem its adjoint $B_\omega^\star$ is a surjection of $c_0(\mathds{Z};X^\star)^\star = \ell^1(\mathds{Z};X^{\star\star})$ onto itself. The computation above shows that $(B_\omega^\star\xi)_n = \xi_n - d(\varphi^n(\omega))^{\star\star}\xi_{n+1}$; in particular, the restriction of $B_\omega^\star$ to $\ell^1(\mathds{Z};X)$ is $D_\omega$.

We show that $D_\omega$ maps $\ell^1(\mathds{Z};X)$ onto itself. Let $\zeta \in \ell^1(\mathds{Z};X)$ and $\varepsilon > 0$. Choose $\xi \in \ell^1(\mathds{Z};X^{\star\star})$ with $B_\omega^\star\xi = \zeta$, choose $M$ such that the truncation $\xi^{(M)}$ of $\xi$ to $[-M,M]$ satisfies $\|\xi - \xi^{(M)}\| \leq \varepsilon$, and put $W = [-M-1, M]$. Then $B_\omega^\star\xi^{(M)}$ vanishes outside $W$, and therefore $\|\zeta - \zeta^W\| \leq c\varepsilon$ and $\|B_\omega^\star\xi^{(M)} - \zeta^W\| \leq c\varepsilon$, where $\zeta^W$ is the truncation of $\zeta$ to $W$ and $c = \|B_\omega\|$. Let $C$ be the subspace of $\ell^1(W;X)$ consisting of all $D_\omega x$, where $x \in \ell^1(\mathds{Z};X)$ vanishes outside $[-M,M]$. We claim that the distance from $\zeta^W$ to $C$ does not exceed $c\varepsilon$. Otherwise, by the Hahn--Banach theorem there is $u = (u_n)_{n \in W}$, $u_n \in X^\star$, $\max_n\|u_n\| = 1$, such that $\langle y, u\rangle = 0$ for all $y \in C$ and $|\langle \zeta^W, u\rangle| > c\varepsilon$. By Goldstine's theorem (see \cite{Go} or \cite[Lemma 3.4, p.~69]{Br}), for every $n$, $|n| \leq M$, there is a net $(x^\alpha_n)$ in $X$ converging to $\xi_n$ in the topology $\sigma(X^{\star\star}, X^\star)$. Since the operators $d(\varphi^n(\omega))^{\star\star}$ are continuous in this topology and extend $d(\varphi^n(\omega))$, for $x^\alpha = (x^\alpha_n)_{|n| \leq M}$ we obtain
\[ 0 = \langle D_\omega x^\alpha, u\rangle \to \langle B_\omega^\star\xi^{(M)}, u\rangle. \]
Hence $|\langle \zeta^W, u\rangle| = |\langle \zeta^W - B_\omega^\star\xi^{(M)}, u\rangle| \leq c\varepsilon$, a contradiction. Consequently the distance from $\zeta$ to the range of $D_\omega$ does not exceed $2c\varepsilon$. Since $\varepsilon$ is arbitrary and the range of $D_\omega$ is closed, $\zeta$ belongs to the range of $D_\omega$. Thus $D_\omega$ is invertible and $\|D_\omega^{-1}\| \leq K$.
\end{proof}

\begin{lemma} \label{l38}
Let $C' = I - \operatorname{diag}(e_n)_n S$ be an operator on $\ell^p(\mathds{Z};X)$ with $\sup_n\|e_n\| \leq \Delta$, invertible with $\|(C')^{-1}\| \leq K$, and let $\gamma > 0$ be such that $(e^\gamma - 1)\Delta K \leq 1/2$. Then $C = (C')^{-1}$ satisfies $\|C[n,m]\| \leq 2Ke^{-\gamma|n-m|}$, $n, m \in \mathds{Z}$.
\end{lemma}

\begin{proof}
Fix $m$, let $J = \operatorname{diag}(e^{-\gamma|n-m|})_n$, and let $C'' = I - \operatorname{diag}(e_n\theta_n)_n S$, where $\theta_n = e^{\gamma(|n-m| - |n+1-m|)}$. Since $|\theta_n - 1| \leq e^\gamma - 1$, we have $\|C'' - C'\| \leq (2K)^{-1}$, so $C''$ is invertible and $\|(C'')^{-1}\| \leq 2K$. Comparing matrix entries we see that $C'J = JC''$. For $x \in X$ let $y = (C'')^{-1}(e_m \otimes x)$. Then $C'(Jy) = J(e_m \otimes x) = e_m \otimes x$, hence $Jy = C(e_m \otimes x)$, i.e., $C[n,m]x = e^{-\gamma|n-m|}y_n$, and $\|C[n,m]x\| \leq 2Ke^{-\gamma|n-m|}\|x\|$.
\end{proof}

The conjugation by exponential weights in the proof of Lemma~\ref{l38} is a discrete version of the passage from $b = \sum a_g \otimes g$ to $b(\xi) = \sum e^{\langle \xi, g\rangle} a_g \otimes g$ in the proof of \cite[Theorem 26.10, pp.~376--377]{AL}. It is also the idea of the Combes--Thomas estimate for resolvents of Schr\"odinger operators \cite{CT}. For banded matrices on $\ell^2$, exponential decay of the entries of the inverse is due to Demko, Moss and Smith \cite{DMS}. Lemma~\ref{l38} is a simple version of these results for operator-valued matrices on $\ell^p(\mathds{Z};X)$.

\begin{lemma} \label{l39}
Let $C'_1 = I - \operatorname{diag}(e_n)_nS$ and $C'_2 = I - \operatorname{diag}(e'_n)_nS$ both satisfy the assumptions of Lemma~\ref{l38} with the same $\Delta$, $K$, $\gamma$, and let $C_i = (C'_i)^{-1}$. If $\|e_n - e'_n\| \leq \eta$ for $|n| \leq M$, then for every $k \in \mathds{Z}$
\[ \|C_1[0,k] - C_2[0,k]\| \leq 4K^2\Big(\eta s_\gamma + \frac{4\Delta e^{-\gamma M}}{1 - e^{-\gamma}}\Big), \quad s_\gamma = \sum_{n \in \mathds{Z}} e^{-\gamma|n|}. \]
\end{lemma}

\begin{proof}
We have $C_1 - C_2 = C_2(C'_2 - C'_1)C_1$ and $C'_2 - C'_1 = \operatorname{diag}(e_n - e'_n)_nS$. By Lemma~\ref{l38} the map $\xi \mapsto (C_2\xi)_0 = \sum_n C_2[0,n]\xi_n$ is given by an absolutely convergent series, hence $(C_1 - C_2)[0,k] = \sum_n C_2[0,n](e_n - e'_n)C_1[n+1,k]$, and
$\|(C_1 - C_2)[0,k]\| \leq 4K^2\sum_n e^{-\gamma|n|}\|e_n - e'_n\|$. It remains to split the sum into $|n| \leq M$ and $|n| > M$, where $\|e_n - e'_n\| \leq 2\Delta$.
\end{proof}

\begin{proposition} \label{p310}
Let $d \in L^\infty_u(\Omega;L(X))$ and assume that $D = I - M_dV$ is invertible on $L^p(X)$, $1 \leq p < \infty$. Then there are $c_k \in L^\infty_u(\Omega;L(X))$ and constants $C, \gamma > 0$ such that $\|c_k\|_\infty \leq Ce^{-\gamma|k|}$ and $D^{-1} = \sum_{k \in \mathds{Z}} M_{c_k}V^k$, the series converging absolutely in operator norm. In particular $D^{-1} \in \mathcal{B}$.
\end{proposition}

\begin{proof}
Let $N$, $K$ and $D_\omega$ be as in Lemma~\ref{l37}, $\Delta = \|d\|_\infty$, and $\gamma$ as in Lemma~\ref{l38}. For $\omega \notin N$ let $C_\omega = D_\omega^{-1}$ and $c_k(\omega) = C_\omega[0,k]$; for $\omega \in N$ let $c_k(\omega) = 0$. Thus the coefficients of $D^{-1}$ are read off from the matrix entries of the orbit operators; this is the counterpart of formula (8.10) in \cite[Theorem 8.11, p.~73]{An}, which expresses the matrix entries of an element in the regular representation through its coefficients.

(i) By Lemma~\ref{l38}, $\|c_k(\omega)\| \leq 2Ke^{-\gamma|k|}$ for every $\omega$.

(ii) A direct computation shows that $D_{\varphi(\omega)}[n,m] = D_\omega[n+1,m+1]$, that is, $D_{\varphi(\omega)} = SD_\omega S^{-1}$. (For points of the same orbit the orbit operators are similar; cf.\ \cite[p.~93]{An} and \cite[Remark 21.4]{AL}.) Hence $C_{\varphi(\omega)} = SC_\omega S^{-1}$, and since $N$ is $\varphi$-invariant, $C_\omega[1,k] = c_{k-1}(\varphi(\omega))$ for $\omega \notin N$. The entry $[0,k]$ of the identity $D_\omega C_\omega = I$ therefore gives
\begin{equation} \label{eqcoef}
c_k(\omega) - d(\omega)c_{k-1}(\varphi(\omega)) = \delta_{k0}I, \quad \omega \notin N.
\end{equation}

(iii) $c_k \in L^\infty_u(\Omega;L(X))$. Let $\delta > 0$. Choose $M$ and $\eta > 0$ such that the right-hand side of the inequality in Lemma~\ref{l39} is at most $\delta$. For every family $\beta = (\beta_n)_{|n| \leq M}$ of elements of $\mathcal{D}_d$ let $H_\beta = \{\omega \notin N : \|d(\varphi^n(\omega)) - \beta_n\| \leq \eta/2 \text{ for } |n| \leq M\}$. These countably many measurable sets cover $\Omega \setminus N$ up to a set of measure zero. Enumerate them, make them pairwise disjoint, choose a point $\omega_j$ in each nonempty set $H'_j$ of the resulting partition, and put $g = c_k(\omega_j)$ on $H'_j$. If $\omega \in H'_j$, then $\|d(\varphi^n(\omega)) - d(\varphi^n(\omega_j))\| \leq \eta$ for $|n| \leq M$, and by Lemma~\ref{l39} applied to $D_\omega$ and $D_{\omega_j}$ we get $\|c_k(\omega) - g(\omega)\| \leq \delta$. Thus $c_k$ is, outside a set of measure zero, a uniform limit of countably valued measurable functions, and by (i) $c_k \in L^\infty_u(\Omega;L(X))$.

(iv) By (i), (iii) and Lemma~\ref{l31}, the series $R = \sum_k M_{c_k}V^k$ converges absolutely in $L(L^p(X))$, so $R \in \mathcal{B}$. By (\ref{eqrules}) and (\ref{eqcoef}),
\[ DR = \sum_k M_{c_k}V^k - \sum_k M_{d\,(c_{k-1}\circ\varphi)}V^k = \sum_k M_{c_k - d\,(c_{k-1}\circ\varphi)}V^k = I. \]
Since $D$ is invertible, $R = D^{-1}$.
\end{proof}

\begin{remark} \label{r32}
Proposition~\ref{p310} is an $L^p(X)$-analogue of the Wiener-type theorems of \cite[Section 26]{AL}. By \cite[Theorem 26.12]{AL} (due to Kurbatov and Lebedev), if $G \cong \mathds{Z}$ and the action is topologically free, then the elements of $C^\star(A,T)$ with exponentially decreasing coefficients form a full (inverse closed) subalgebra. In particular, the inverse of an invertible finite sum $\sum a_k T^k$ has exponentially decreasing coefficients. The proof in \cite{AL} is based on the Bochner--Phillips theorem and on the $C^\star$-structure, and does not apply in our setting.
\end{remark}

\subsection{The spectral projection and its representative}

\begin{corollary} \label{c311}
Under the assumptions of Theorem~\ref{t4}, $P \in \mathcal{B}$.
\end{corollary}

\begin{proof}
If $|\lambda| = 1$, then $\lambda \notin \sigma(T)$, and $(\lambda - T)^{-1} = \lambda^{-1}(I - M_{A/\lambda}V)^{-1} \in \mathcal{B}$ by Proposition~\ref{p310}. The map $\lambda \mapsto (\lambda - T)^{-1}$ is norm continuous on $\mathds{T}$, so the Riemann sums of
\[ P = \frac{1}{2\pi i}\oint_{\mathds{T}} (\lambda - T)^{-1}\, d\lambda \]
belong to $\mathcal{B}$ and converge to $P$ in norm. Since $\mathcal{B}$ is closed, $P \in \mathcal{B}$ (cf.\ \cite[p.~81]{An}).
\end{proof}

\begin{lemma} \label{l312}
Let $b \in L^\infty_u(\Omega;L(X))$ and assume that $M_b$ is a projection on $L^p(X)$. Then $b$ has a representative such that $b(\omega)^2 = b(\omega)$ and $\|b(\omega)\| \leq \|M_b\|$ for every $\omega \in \Omega$.
\end{lemma}

\begin{proof}
By Lemma~\ref{l31}, $\|b\|_\infty = \|M_b\|$, and $M_{b^2 - b} = M_b^2 - M_b = 0$ gives $\|b^2 - b\|_\infty = 0$. Hence the set of all $\omega$ such that $b(\omega)^2 \neq b(\omega)$ or $\|b(\omega)\| > \|M_b\|$ is contained in a set of measure zero, and it suffices to put $b(\omega) = 0$ on this set.
\end{proof}

Theorem~\ref{t4} is stated for the splitting of $\sigma(T)$ by the unit circle. By scaling, it applies to every circle centered at $0$ that does not meet $\sigma(T)$, and, since $\sigma(T)$ is rotation invariant, to every clopen subset of $\sigma(T)$.

\begin{corollary} \label{c313}
Let $(\Omega, \Sigma, \mu)$, $p$, $\varphi$, $X$, $A$ and $T = M_AV$ be as in Theorem~\ref{t4}, and assume that $T$ is invertible.
\begin{enumerate}[(a)]
\item If $r > 0$ and $\sigma(T) \cap \{|\lambda| = r\} = \emptyset$, then the conclusion of Theorem~\ref{t4} holds for the spectral projection $\Pi_r$ of $T$ corresponding to $\sigma(T) \cap \{|\lambda| < r\}$.
\item The spectrum $\sigma(T)$ is rotation invariant, and the conclusion of Theorem~\ref{t4} holds for the spectral projection $P'$ of $T$ corresponding to an arbitrary nonempty clopen subset $\sigma'$ of $\sigma(T)$.
\end{enumerate}
\end{corollary}

\begin{proof}
(a) The operator $r^{-1}T = M_{A/r}V$ satisfies the assumptions of Theorem~\ref{t4}, and $\Pi_r$ is the spectral projection of $r^{-1}T$ corresponding to the part of $\sigma(r^{-1}T)$ in $\mathds{U}$.

(b) By (\ref{eqrules}), $Tf = (f \circ \varphi)T$ for $f \in L^\infty(\mu)$. Hence, if we consider $L^p(X)$ as a Banach $C(\mathcal{K})$-module, where $\mathcal{K}$ is the Gelfand compact of $L^\infty(\mu)$, then $T$ is a weighted composition operator with respect to the homeomorphism $\psi$ of Remark~\ref{rfrolik}. Since $\varphi$ is aperiodic, $\psi$ has no periodic points (Remark~\ref{rfrolik}), and by \cite[Corollary 12.30]{AAK} the spectrum $\sigma(T)$ is rotation invariant.

Now, as in Step 2 of the proof of Corollary~\ref{c1}, there are $0 < s_1 < t_1 < \dots < s_k < t_k$ such that the circles $\{|\lambda| = s_i\}$, $\{|\lambda| = t_i\}$ do not meet $\sigma(T)$ and $P' = \sum_{i=1}^k (\Pi_{t_i} - \Pi_{s_i})$. By (a), $P' = M_b$ with $b = \sum_{i=1}^k (\Pi_{t_i}(\cdot) - \Pi_{s_i}(\cdot)) \in L^\infty_u(\Omega;L(X))$, where $\Pi_r(\cdot)$ denotes the norm-measurable function representing $\Pi_r$, and it remains to apply Lemma~\ref{l312}.
\end{proof}

\begin{remark} \label{r6}
(1) By a similar but longer argument one can show that $P(\omega)A(\omega) = A(\omega)P(\varphi(\omega))$ for all $\omega$ outside a single set of measure zero, and that the products $a(\omega)a(\varphi(\omega)) \cdots a(\varphi^{n-1}(\omega))$, where $a = \rho^{1/p}A$, have an exponential dichotomy with respect to the splitting $X = P(\omega)X \oplus (I - P(\omega))X$, again outside a single set of measure zero and for all vectors of $X$ simultaneously. This is the measurable counterpart of the equivalence between hyperbolicity of a weighted composition operator and hyperbolicity (exponential dichotomy) of the associated linear extension, see \cite[Theorems 9.2 and 9.3]{An}, \cite[Theorems 19.13 and 21.15]{AL} and the references in \cite[Notes to Sections 19 and 21]{AL} (Mather, Latushkin--Stepin).

(2) We do not know whether Theorem~\ref{t4} remains true if $A$ is only strongly measurable, if $L^p$ is replaced by a general K\"othe space $E$, or if $p = \infty$. In the proof, norm-measurability of $A$ is used in Lemmas~\ref{l31}, \ref{l36} and in part (iii) of the proof of Proposition~\ref{p310}; the structure of $L^p$, through the fact that $V$ is an isometry, is used in Lemma~\ref{l36}; and $p < \infty$ is used in Lemma~\ref{l37}: for $p = \infty$, Remark~\ref{rinfty} gives lower bounds for the orbit operators only on $c_0(\mathds{Z};X)$, and the dual of $L^\infty(X)$ is not a space of functions. The finiteness of $\mu$ is not needed: $\sigma$-finiteness is used only to choose sets of finite positive measure in Lemmas~\ref{l31} and \ref{l36}, and in the duality argument of Lemma~\ref{l37}.

(3) Examples with infinite measure. Let $\Omega = \mathds{R}$ with Lebesgue measure and $\varphi(t) = t+1$. Then $\varphi$ is measure-preserving and aperiodic, and $T$ is the weighted translation $(Tx)(t) = A(t)x(t+1)$ on $L^p(\mathds{R};X)$. This is the situation considered by Chicone and Latushkin \cite[Section 4.1]{CL}, where the coefficients are assumed to be strongly continuous; Theorem~\ref{t4} assumes norm-measurability instead, and neither of these conditions implies the other. If $\Omega = \mathds{Z}$ with the counting measure and $\varphi(n) = n+1$, then every one-point set is $N$-separated for every $N$, so $\varphi$ is aperiodic, and the spectral projections of weighted composition operators on $\ell^p(\mathds{Z};X)$ (in this case, the classical bilateral weighted shifts) are diagonal operators. In this case, however, this follows directly from Remark~\ref{ratomic}. An example with a transformation that does not preserve the measure is $\Omega = \mathds{R}$ with Lebesgue measure and $\varphi(t) = 2t$. Then $\varphi$ is nonsingular and aperiodic, $\rho = 2$, and Theorem~\ref{t4} applies to the operators $(Tx)(t) = A(t)x(2t)$ on $L^p(\mathds{R};X)$ with $A \in L^\infty_u(\mathds{R};L(X))$.

(4) Nonseparable spaces that are not complemented in the bidual. If $X$ is separable or complemented in $X^{\star\star}$, decomposability of the spectral projections already follows from Corollaries~\ref{c1} and \ref{c1a}, and Theorem~\ref{t4} adds the norm-measurability of $P(\cdot)$ and the fact that $P(\omega)$ is a projection for \emph{every} $\omega$ (compare Remark~\ref{r1a}). The main case of interest in Theorem~\ref{t4} is therefore that of a nonseparable space $X$ that is not complemented in $X^{\star\star}$. A typical example is $X = c_0(\Gamma)$ with $\Gamma$ uncountable. Then $X^{\star\star} = \ell^\infty(\Gamma)$, and $c_0(\Gamma)$ is not complemented in $\ell^\infty(\Gamma)$. Indeed, let $\Gamma_0 \subset \Gamma$ be countably infinite. If $Q$ were a bounded projection of $\ell^\infty(\Gamma)$ onto $c_0(\Gamma)$, then $RQJ$, where $J : \ell^\infty(\Gamma_0) \to \ell^\infty(\Gamma)$ is the extension by zero and $R : c_0(\Gamma) \to c_0(\Gamma_0)$ is the restriction, would be a bounded projection of $\ell^\infty(\Gamma_0) \cong \ell^\infty$ onto $c_0(\Gamma_0) \cong c_0$, which contradicts Phillips' theorem \cite{Ph}. Thus on $L^p(\Omega; c_0(\Gamma))$ Theorems~\ref{t1} and \ref{t2} do not apply, and Remark~\ref{r1} gives only a representation of the operators commuting with $L^\infty(\mu)$ by functions with values in $L(c_0(\Gamma), \ell^\infty(\Gamma))$. On the other hand, Theorem~\ref{t4} applies to every $T = M_AV$ with $A \in L^\infty_u(\Omega; L(c_0(\Gamma)))$ satisfying its assumptions, for instance with $A = \sum_j \chi_{G_j}B_j$, where $G_j \in \Sigma$ are pairwise disjoint and $B_j \in L(c_0(\Gamma))$, $\sup_j\|B_j\| < \infty$ (e.g. diagonal operators $(B_jy)(\gamma) = \beta_j(\gamma)y(\gamma)$ with $\beta_j \in \ell^\infty(\Gamma)$), or with a uniform limit of such functions. In this case the spectral projections are given by norm-measurable functions with values in the projections of $L(c_0(\Gamma))$.
\end{remark}

\section*{Acknowledgments}
The authors used Claude, an AI model developed by Anthropic, as an assistant during the preparation of this paper. Claude's contributions were the following.

\emph{Comparison with the literature.} Claude compared our arguments in detail with the books \cite{An} and \cite{AL}, identified the places where similar ideas appear there (the conditional expectation and property $(*)$, topologically and metrically free actions, the commutation argument for spectral subspaces, trajectorial representations, Wiener-type theorems, and hyperbolicity of linear extensions), and supplied the exact references in Sections~\ref{s3} and~\ref{s4}, including Remark~\ref{r0} comparing Corollary~\ref{c2} with \cite[Theorem 9.5]{An}. Claude also suggested the references \cite{CT, DMS, LS}, and found or completed the bibliographic data of \cite{CRS, dPR, vN31, vN49}.

\emph{Checking of proofs.} Claude pointed out that in the case when $X$ is complemented in $X^{\star\star}$ the original formulation claimed more than Theorem~\ref{t2} provides; this led to the separation of Corollaries~\ref{c1} and~\ref{c1a} and to Remark~\ref{r1a}. Claude observed that for purely atomic measures our results hold for arbitrary $X$ (Remark~\ref{ratomic}), and that the example with $\ell^p(\mathds{Z};X)$ in Remark~\ref{r6} is therefore elementary. Claude also found gaps in the original proofs of Corollaries~\ref{c2} and~\ref{c3} and suggested the reference \cite{Ne}, and corrected several statements in the introductory part of Section~\ref{s4} that overstated what follows from the results of Section~\ref{s3}. Claude pointed out that the proofs of Theorem~\ref{t2} and Corollary~\ref{c1} needed more detail, proposed the extension to arbitrary clopen subsets of the spectrum (Step 2 of the proof of Corollary~\ref{c1} and Corollary~\ref{c313}), suggested discussing the example $X = c_0(\Gamma)$ in Remark~\ref{r6}(4), and worked out the details of Remark~\ref{rfrolik} following the argument of the authors based on Frol\'{\i}k's theorem.

\emph{Exposition.} Claude restructured the presentation of the proof of Theorem~\ref{t4} (the outline in Subsection~\ref{ssoutline}, Corollary~\ref{c311} and Lemma~\ref{l312}), rewrote the definition of the orbit operators, changed the notation in the description of the results of \cite{AL, An} to avoid conflicts with ours, worked out the details of the extension of Theorem~\ref{t4} from measure-preserving to nonsingular transformations proposed by the authors, proposed and worked out the extension of Theorem~\ref{t4} to the case $p = 1$, drafted the introduction and the list of notation, and edited the text and the bibliography.

All mathematical results and their final formulations are the responsibility of the authors.


\begin{thebibliography}{99}
\bibitem{AAK} Y.~A. Abramovich, E.~L. Arenson, and A.~K. Kitover, \emph{Banach $C(K)$-modules and operators preserving disjointness}, Pitman Research Notes in Mathematics Series, vol.~277, Longman, Harlow, 1992.
\bibitem{An} A. Antonevich, \emph{Linear functional equations. Operator approach}, Operator Theory: Advances and Applications, vol.~83, Birkh\"auser, Basel, 1996.
\bibitem{AL} A. Antonevich and A. Lebedev, \emph{Functional differential equations. I. $C^\star$-theory}, Pitman Monographs and Surveys in Pure and Applied Mathematics, vol.~70, Longman, Harlow, 1994.
\bibitem{Be} E. Behrends, \emph{$M$-structure and the Banach--Stone theorem}, Lecture Notes in Mathematics, vol.~736, Springer, Berlin, 1979.
\bibitem{Bl} B. Blackadar, \emph{Operator algebras. Theory of $C^\star$-algebras and von Neumann algebras}, Encyclopaedia of Mathematical Sciences, vol.~122, Springer, Berlin, 2006.
\bibitem{Br} H. Brezis, \emph{Functional analysis, Sobolev spaces and partial differential equations}, Universitext, Springer, New York, 2011.
\bibitem{CRS} J.~M. Calabuig, J. Rodr\'{\i}guez, and E.~A. S\'anchez-P\'erez, \emph{Multiplication operators in K\"othe--Bochner spaces}, J. Math. Anal. Appl. \textbf{373} (2011), no.~1, 316--321.
\bibitem{CV} C. Castaing and M. Valadier, \emph{Convex analysis and measurable multifunctions}, Lecture Notes in Mathematics, vol.~580, Springer, Berlin, 1977.
\bibitem{CL} C. Chicone and Y. Latushkin, \emph{Evolution semigroups in dynamical systems and differential equations}, Mathematical Surveys and Monographs, vol.~70, American Mathematical Society, Providence, RI, 1999.
\bibitem{CT} J.-M. Combes and L. Thomas, \emph{Asymptotic behaviour of eigenfunctions for multiparticle Schr\"odinger operators}, Comm. Math. Phys. \textbf{34} (1973), 251--270.
\bibitem{DMS} S. Demko, W.~F. Moss, and P.~W. Smith, \emph{Decay rates for inverses of band matrices}, Math. Comp. \textbf{43} (1984), no.~168, 491--499.
\bibitem{Di} J. Dixmier, \emph{Von Neumann algebras}, North-Holland Mathematical Library, vol.~27, North-Holland, Amsterdam, 1981.
\bibitem{DKO} H. Duru, A. Kitover, and M. Orhon, \emph{Multiplication operators on vector-valued function spaces}, Proc. Amer. Math. Soc. \textbf{141} (2013), no.~10, 3501--3513.
\bibitem{Fr} Z. Frol\'{\i}k, \emph{Fixed points of maps of extremally disconnected spaces and complete Boolean algebras}, Bull. Acad. Polon. Sci. S\'er. Sci. Math. Astronom. Phys. \textbf{16} (1968), 269--275.
\bibitem{Go} H.~H. Goldstine, \emph{Weakly complete Banach spaces}, Duke Math. J. \textbf{4} (1938), 125--131.
\bibitem{IT} A. Ionescu Tulcea and C. Ionescu Tulcea, \emph{Topics in the theory of lifting}, Ergebnisse der Mathematik und ihrer Grenzgebiete, vol.~48, Springer, Berlin, 1969.
\bibitem{KA} L.~V. Kantorovich and G.~P. Akilov, \emph{Functional analysis}, 2nd ed., Pergamon Press, Oxford, 1982.
\bibitem{LS} Yu.~D. Latushkin and A.~M. Stepin, \emph{Weighted translation operators and linear extensions of dynamical systems}, Uspekhi Mat. Nauk \textbf{46} (1991), no.~2, 85--143; English transl., Russian Math. Surveys \textbf{46} (1991), no.~2, 95--165.
\bibitem{Ma} D. Maharam, \emph{On a theorem of von Neumann}, Proc. Amer. Math. Soc. \textbf{9} (1958), 987--994.
\bibitem{Ne} J.~D. Newburgh, \emph{The variation of spectra}, Duke Math. J. \textbf{18} (1951), 165--176.
\bibitem{vN31} J. von Neumann, \emph{Algebraische Repr\"asentanten der Funktionen ``bis auf eine Menge vom Ma{\ss}e Null''}, J. Reine Angew. Math. \textbf{165} (1931), 109--115.
\bibitem{vN49} J. von Neumann, \emph{On rings of operators. Reduction theory}, Ann. of Math. (2) \textbf{50} (1949), 401--485.
\bibitem{dPR} B. de Pagter and W.~J. Ricker, \emph{Algebras of multiplication operators in Banach function spaces}, J. Operator Theory \textbf{42} (1999), 245--267.
\bibitem{Ph} R.~S. Phillips, \emph{On linear transformations}, Trans. Amer. Math. Soc. \textbf{48} (1940), 516--541.
\bibitem{Ry} R.~A. Ryan, \emph{Introduction to tensor products of Banach spaces}, Springer Monographs in Mathematics, Springer, London, 2002.
\end{thebibliography}
\end{document}